\documentclass[11pt,a4paper,reqno]{amsart}
\usepackage{color,latexsym,amsfonts,amssymb,bbm,comment}
\usepackage{hyperref}
\usepackage{amsmath,cite}
\usepackage{amsthm}
\usepackage{fullpage}
\usepackage{graphicx} 
\usepackage{tikz} 
\usepackage{caption,subcaption}

\usepackage{pgfplots}

\newcommand{\wt}{\widetilde}

\newcommand{\E}{\mathbb{E}} 
\newcommand{\Z}{\mathbb{Z}}
\newcommand{\R}{\mathbb{R}}

\newcommand\supp{\mathop{\rm{ supp}}}

\def\eq{\begin{equation}}
\def\en{\end{equation}}

\newtheorem{theorem}{Theorem}[section]
\newtheorem{corollary}[theorem]{Corollary}
\newtheorem{lemma}[theorem]{Lemma}
\newtheorem{proposition}[theorem]{Proposition}
\theoremstyle{definition}
\newtheorem{definition}[theorem]{Definition}

\def\I{{\Bbb I}}
\def\DDD{{\mathcal D}}

\def\e{{\varepsilon}}

\def\D{\Delta}
\def\a{\alpha}

\def\b{\beta}
\def\MM{{\cal M}}

\def\e{\varepsilon}

\def\phi{\varphi}
\def\g{\gamma}
\def\la{\lambda}
\def\k{\kappa}
\def\r{\rho}
\def\de{\delta}

\def\x{\xi}

\def\D{\Delta}

\def\G{\Gamma}

\def\P{{\Phi}}

\def\T{\T}

\def\MM{\mathcal M}

\def\nn{{\mathfrak{n}}}

\def\V|{{\Vert}}
\def\LLL{{\mathfrak L}}

\def\d{{\rm d}}
\def\E{\mathbb{E}}
\def\I{\mathcal{I}}

\def\one{\mathbbmss{1}}
\def\mc{\mathcal}

\def\ms{\mathsf}

\def\one{\mathbbmss{1}}
\def\P{\mathbb{P}}
\def\R{\mathbb{R}}
\def\Z{\mathbb{Z}}

\def\TF{t_{\ms f}}
\def\XTi{X_i}

\def\Mac{\MM_{\ms{ac}}}

\def\Lla{L_{\la}}
\def\rla{r_{\la}}
\def\rhola{\rho_{\la}}

\def\muR{\mu_{\ms{R}}}
\def\nuR{\nu_{\ms{R}}}
\def\lla{l_\la}

\def\mula{\mu_\la}

\def\itf{[0,\TF]}

\def\itfiww{\itf^2  \times W^2\times [0,1]}

\def\nns{\nn^*}
\def\nnsp{\nn^{*,+}}
\def\nnsm{\nn^{*,-}}

\def\mus{\mu^*}
\def\muST{\mu^{\ms{t}}_{\ms{T}} \otimes \mu^{\ms{s}}_{\ms{T}}}
\def\muT{\mu_{\ms{T}}}

\def\musp{\mu^{\ms{s}}_{\ms{T}}}
\def\mut{\mu^{\ms{t}}_{\ms{T}}}

\def\LLLd{\LLL^\de_\la}

\def\rla{r_\la}

\def\Ss{S}
\def\Tt{T}
\def\Zz{Y^{\ms{sel}}}

\newcommand{\ao}[1]{a^{\ms{oc}#1}}
\newcommand{\ac}{a^{\ms{crit}}}
\newcommand{\addo}[2]{a^{#1,\ms{oc}#2}}
\newcommand{\addi}[1]{a^{#1,\ms{idle}}}

\newcommand{\addc}[2]{a^{#1,\ms{crit}#2}}

\def\ai{a^{\ms{idle}}}
\def\bi{b^{\ms{idle}}}

\def\acn{a^{\ms{crit}'}}
\newcommand{\addcn}[2]{a^{#1,\ms{crit}'#2}}

\def\Lmin{L^{\ms{min}}_\la}
\def\Ycrit{Y^{\ms{crit}}}
\def\Memp{{\mathcal{M}}_{\ms{emp}}}
\def\Zld{Z^{\la,\de}}
\def\Lld{L_\la^{\de}}

\keywords{large deviations, entropy, capacity, relay}

\begin{document}
\author{Christian Hirsch}
\address[Christian Hirsch]{Mathematisches Institut, Ludwig-Maximilians-Universit\"at M\"unchen, 80333 Munich, Germany}
\email{hirsch@math.lmu.de} 
\author{Benedikt Jahnel}
\address[Benedikt Jahnel]{Weierstrass Institute for Applied Analysis and Stochastics, Mohrenstra\ss e 39, 10117 Berlin, Germany}
              \email{benedikt.jahnel@wias-berlin.de}

\title{Large deviations for the capacity in dynamic spatial relay networks}

\date{\today}

\begin{abstract}
We derive a large deviation principle for the space-time evolution of users in a relay network that are unable to connect due to capacity constraints. The users are distributed according to a Poisson point process with increasing intensity in a bounded domain, whereas the relays are positioned deterministically with given limiting density. 
The preceding work on capacity for relay networks by the authors describes the highly simplified setting where users can only enter but not leave the system. In the present manuscript we study the more realistic situation where users leave the system after a random transmission time. For this we extend the point process techniques developed in the preceding work thereby showing that they are not {limited} to settings with strong monotonicity properties. 
\end{abstract}

\maketitle

\section{Introduction and main results}
Loss networks are classical {models} in mathematical queueing theory designed for capacity-constrained scenarios, where network participants can leave the system without being served, see for example~\cite{Ke91}. The underlying Markovian dynamics is challenging from a mathematical point of view and a substantial amount of research was performed to establish classical limiting statements such as propagation of chaos or central limit theorems, see~\cite{GrMe93,Gr01}. 

\medskip
{In}~\cite{gramMel1,gramMel2}, a large deviation analysis of loss networks was carried out in a mean-field setting where connections are formed disregarding geometry. In the presence of geometry, the models for random networks become substantially more complex to analyze, see for example~\cite{chatterjee2014localization}. A first step to investigate spatial loss networks was taken in~\cite{wireless3}, {in a situation where}
transmitters are distributed in a bounded domain via a Poisson point process with increasing intensity. Deterministic relays are {additionally} placed in the domain and users try to connect to the relays based on their positions in space. Transmissions are attempted at random times and once communication is established, the channel stays active and is blocked for other users for the remaining time. As a consequence, the system exhibits strong monotonicity properties which simplify the mathematical analysis. 

\medskip
In the present work, we show that the point-process techniques mentioned above are applicable in a broader context, in the sense that they do not rely on these monotonicity assumptions. In particular, we are able to derive large deviation results also in the case where transmissions are stopped at random times. The introduction of finite transmission times leads to more dependencies, which have to be controlled in our approximation approach. 
To illustrate this, consider the effect of a small perturbation in the behavior of a single user with a large transmission time. If the user chooses a different relay location, all other users that previously selected this relay could be affected. 
Next, {let us} provide a precise description of the model.

\medskip
First, we present a detailed description of the network model {which is an extension of the one} introduced in~\cite{wireless3}. Let $W\subset\R^d$ be a compact domain with boundaries of vanishing Lebesgue measure. We denote by $Y^\la=(y_i)_{ i\le n_\la}$ a collection of $n_\la$ fixed relays for which the empirical distribution 
$$l_{\la} = \la^{-1} \sum_{ i \le n_\la}\de_{y_i}$$
converges weakly to some probability measure $\mu_{\ms R}$ on $W$ {as $\la$ tends to infinity}. Further, there will be transmitters distributed according to a Poisson point process $X^\la$ in $W$. Its intensity measure is of the form $\la\mu^{\ms s}_{\ms{T}}$ with $\la>0$ and $\mu^{\ms s}_{\ms T} \in \MM(W)$ a finite Borel measure on $W$. We assume that $\mu^{\ms s}_{\ms T} \in \MM(W)$ is absolutely continuous w.r.t.~the Lebesgue measure. Each transmitter $\XTi$ starts sending data at a random time $\Ss_i \in \itf$. In contrast to~\cite{wireless3}, it stops the transmission at another random time $\Tt_i \in \itf$. We assume that the bivariate random variables $\{(S_i,T_i)\}_{i\ge1}$ are iid with a distribution $\muT$ that is absolutely continuous w.r.t.~the Lebesgue measure on $\itf^2$.

\medskip
At time $\Ss_i$ the transmitter $X_i$ selects a relay $\Zz_i\in Y^\la$ randomly according to the \textit{preference kernel}

\begin{align}\label{Kappa}
\k(\Zz_{i}| \XTi)=\frac{\k(\XTi, \Zz_{i})}{\sum_{y_k\in Y^\la}\k(\XTi, y_{k})}.
\end{align}
If the chosen relay is available, then $X_i$ holds the connection up to time $T_i$. This chosen relay $\Zz_i$ is then blocked in the time interval $[S_i, T_i]$ and not available for other transmitters. In the selection process, transmitters are not aware of the status of relays. In particular, they might choose a relay which is already occupied. We then call the transmitter \textit{frustrated}.

\medskip
In order to assess network quality, it is essential for a network operator to answer the following questions.
\begin{enumerate}
	\item What is the probability that an atypically large proportion of transmitters is frustrated?
	\item How do location or data-transmission time influence the frustration risk?
\end{enumerate}

We answer these questions by investigating the \emph{random measure of frustrated transmitters} 
\begin{align}\label{Busy_Process}
	\G^\la=\frac{1}{\la}\sum_{i \ge 1}\one\{\Zz_i(S_i) = 1\}\de_{(S_i, T_i, X_i)},
\end{align}
where $\Zz_i:\, [0,\TF]\to\{0,1\}$ denotes the function taking the value 1 if and only if $\Zz_i$ is occupied at time $t \le \TF$.

\begin{figure}[!htpb]
\centering
\begin{tikzpicture}[xscale=0.7,yscale=0.7]

\foreach \x in {0,1,2}
\foreach \y in {0,1}
{
\draw[black] (\x*4,\y*4) -- (\x*4+4,\y*4) -- (\x*4+4,\y*4+4) -- (\x*4,\y*4+4) -- (\x*4,\y*4);
\fill[color=black] (\x*4+2,\y*4+2) circle(0.25);
}
\node at (0.4,4.4) {1};
\node at (4.4,4.4) {2};
\node at (8.4,4.4) {3};
\node at (0.4,0.4) {4};
\node at (4.4,0.4) {5};
\node at (8.4,0.4) {6};

%\draw[dashed] (\x*4+3,\y*4+2.3)--(\x*4+2,\y*4+2);
%\draw[dashed] (\x*4+1,\y*4+1)--(\x*4+2,\y*4+2);
%\draw[dashed] (\x*4+1.3,\y*4+3.3)--(\x*4+2,\y*4+2);
%\fill[color=red] (\x*4+3,\y*4+2.3) circle(0.2);
%\fill[color=green] (\x*4+1,\y*4+1) circle(0.2);
%\fill[color=red] (\x*4+1.3,\y*4+3.3) circle(0.2);

\draw[dashed] (4+3,4+2.3)--(4+2,4+2);
\fill[color=green] (4+3,4+2.3) circle(0.2);

\draw[dashed] (8+3,4+2.3)--(8+2,4+2);
\draw[dashed] (8+1,4+1)--(8+2,4+2);
\fill[color=green] (8+3,4+2.3) circle(0.2);
\fill[color=red] (8+1,4+1) circle(0.2);

\draw[dashed] (3,2.3)--(2,2);
\draw[dashed] (1,1)--(2,2);
\draw[dashed] (1.3,3.3)--(2,2);
\fill[color=green] (3,2.3) circle(0.2);
\fill[color=red] (1,1) circle(0.2);
\fill[color=red] (1.3,3.3) circle(0.2);

\draw[dashed] (4+1,1)--(4+2,2);
\draw[dashed] (4+1.3,3.3)--(4+2,2);
\fill[color=red] (4+1,1) circle(0.2);
\fill[color=red] (4+1.3,3.3) circle(0.2);

\draw[dashed] (8+1.3,3.3)--(8+2,2);
\fill[color=red] (8+1.3,3.3) circle(0.2);

\end{tikzpicture}
\caption{Collection of three transmitters (green and red) trying to communicate with one relay (black). The transmitters start and stop sending data in time steps $1,\dots,6$. Only the first transmitter (green) can establish a connection. Later transmitters (red) are unable to connect and become frustrated.}
\label{Fig}
\end{figure}
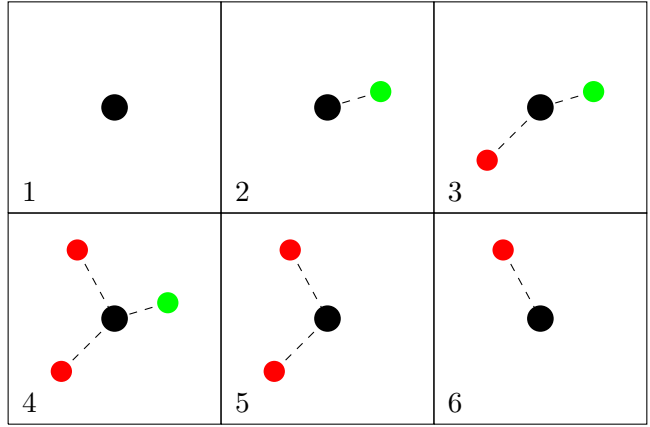

\subsection{The non-spatial case}
First, assume $\k \equiv 1$. That is, transmitters choose relays uniformly at random. Then, as in~\cite{wireless3}, the relay choice is encoded in a uniform random variable on $[0,1]$. More precisely, we attach an independent and uniform random variables $U_i \in [0,1]$ to the transmitter located at $X_i \in W$ and consider the empirical measure for the transmitters given by 
\begin{align*}
	\Lla = \la^{-1}\sum_{i \ge 1}\de_{(S_i, T_i, X_i, U_i)}.
\end{align*}
Note that $\Lla$ is a finite measure on the space $V = \itf^2 \times W \times [0,1]$.
We claim that $\Lla$ is sufficiently rich to describe the random measure of frustrated transmitters. Loosely speaking that is because of the following. If a transmitter arrives at time $t \in \itf$ and at that time $a \ge 0$ relays are already occupied, then, with probability $a / n_\la$, the transmitter selects an occupied relay and therefore becomes frustrated.
To make this precise, we introduce the evolution of the number of occupied relays $\la \wt{B}^\la$ 
%with exit time $\d t'$ and location $\d x$ 
via the time-integral equation
\begin{align}\label{DGL_Empi}
	\wt{B}^\la_t = \int_0^t L_\la(\d s, [t, \TF], W, [0, 1 - \wt{B}^\la_{s-}/r_\la])
\end{align}
where $r_\la=\la^{-1}n_\la$. We will see in Proposition~\ref{markRepLem} that, in distribution, the random measure of frustrated transmitters $\G^\la$ can be represented as 
\begin{align}\label{frustUserDefEmp}
	\tilde{\G}^\la(\d s, \d t, \d x) =  L_\la(\d s, \d t, \d x, [1 - \wt{B}^\la_{s-}/r_\la, 1]).
\end{align}

\medskip
To understand the high-density limit $\la\uparrow\infty$, we need to work with an analogue of equation~\eqref{DGL_Empi} for measures $\nu\in\MM=\MM(V)$ which are absolutely continuous w.r.t.~the measure
$$\muT = \muST \otimes {\bf U}([0,1]),$$
i.e., $\nu\in\Mac(\muT)=\{\nu'\in\MM:\, \nu'\ll\muT\}$. To that end, we investigate the integral equation
\begin{align}\label{DGL_Gen}
	\b_t = \int_0^t \nu(\d s, [t, \TF], W, [0, 1 - \b_{s-}/r])
\end{align}
where $r>0$. If $\nu\in\Mac(\muT)$ then, as shown in Proposition~\ref{approxScalProp}, we can construct a solution $\b_t(\nu,r)$ for \eqref{DGL_Gen} and define
\begin{align}
\label{frustUserDef}
\g(\nu,r)(\d s, \d t,\d x) = \nu(\d s, \d t, \d x, [1 - \b_s(\nu, r), 1]).
\end{align}
To state the main result of this section, we recall the definition of the relative entropy
$$h(\nu|\mu) = \int \log \frac{\d\nu}{\d\mu} \d\nu - \nu(V) + \mu(V)$$
if $\nu\in\Mac(\mu)$ and $h(\nu|\mu)=\infty$ otherwise. Further, recall the $\tau$-topology on $\MM$ where the associated convergence is tested on bounded and measurable functions, see~\cite[Section 6.2]{dz98}.

\begin{theorem}\label{LDP_NoSpatial}
	The family of random measures $\{\G^\la\}_\la$ satisfies the large deviation principle in the $\tau$-topology 
	with good rate function given by $I(\g) = \inf_{\nu \in \MM: \, \g(\nu,\muR(W)) = \g} h(\nu|\mu_{\ms T})$.
\end{theorem}

The main idea for the proof is to introduce approximating trajectories using a temporal discretization which will allow us to apply the contraction principle. 
\subsection{The spatial case}

The process of transmitter requests to a relay at location $\d y$ is a Poisson point process $Z^\la$ 
on $\hat V=\itf^2\times W^2$ 
with intensity measure $\la\mu(l_\la)$ where 
\begin{align*}
	\mu(l_\la)(\d s, \d t ,\d x, \d y) = \k_{l_\la}(\d y|x)(\muST)(\d s, \d t,  \d x)
\end{align*}
and 
\begin{align}\label{Kappa_Index}
\k_{l_\la}(\d y| x) = \k(y| x)l_\la(\d y).
\end{align}
As in~\cite{wireless3} we assume that 
\begin{enumerate}
    \item $\k_\infty = \sup_{x, y \in W}\k(x, y) < \infty$,
    \item the preference kernel $\k$ is jointly continuous $\mu^{\ms s}_{\ms T}\otimes\mu_{\ms R}$-almost everywhere,  and 
    \item for all $x\in W$ there exists $y\in W$ such that $\k(x, y)>0$, $y\in\supp(\muR)$ and $(x, y)$ is a continuity point of $\k$.
\end{enumerate}

\medskip
As in the non-spatial case, the random measure of frustrated transmitters can be described as a function of the empirical measure of the Poisson point process with intensity measure $$\muT(\muR) = \mu(\muR) \otimes {\bf U}([0,1])$$
on the extended state space $V' = \itf^2 \times W^2 \times [0, 1]$.
In the large deviation regime as $\la\uparrow\infty$, this measure can be distorted into another measure $\nn \in \MM'=\MM(V')$ which is absolutely continuous to $\muT(\muR)$. We define $\nn_y$ to be the measure of transmitters choosing a relay at $y$, i.e.,
\begin{align}
\nn(\d s, \d t, \d x, \d y,\d u) = \nn_y(\d s,\d t, \d x, \d u) \muR(\d y).
\end{align}
Then, using $\nn$ as a driving measure, equation \eqref{frustUserDef} becomes 
\begin{align}
\label{frustUserDefSp}
\g(\nn)(\d s, \d t,\d x) = \int_W \nn(\d s, \d t, \d x, \d y, [1 - \b_s(\nn_y, 1), 1]),
\end{align}
where the integration is performed w.r.t.~$\d y$.
As in Theorem~\ref{LDP_NoSpatial}, the function $\nn \mapsto \g(\nn)$ plays the r\^ole of the contraction mapping appearing in the rate function associated with the LDP for $\G^\la$. {We now present our second main result.}

\begin{theorem}\label{LDP_Spatial}
The family of random measures $\{\G^\la\}_\la$ satisfies the LDP in the $\tau$-topology with good rate function given by $I(\g)=\inf_{\nn \in \MM':\, \g(\nn) = \g} h(\nn|\muT(\mu_{\ms R}))$.
\end{theorem}

\subsection{Organization of the manuscript}In Section~\ref{Outline_One} (respectively Section~\ref{Outline_Two}) we present the proof of Theorem~\ref{LDP_NoSpatial} (respectively Theorem~\ref{LDP_Spatial}) via a series of propositions. The details of the proofs for these propositions is then presented in Section~\ref{thm1Sec} (respectively Section~\ref{thm2Sec}).

\section{Outline of proof for Theorem~\ref{LDP_NoSpatial}}\label{Outline_One}
A first idea for a proof of Theorem~\ref{LDP_NoSpatial} would be to represent the random measure of frustrated transmitters $\G^\la$ as a continuous functional of the marked Poisson point process $\Lla$. The desired large deviation principle could then be recovered from Sanov's theorem with the help of the contraction principle. 
However, $\G^\la$ is given as the solution of equation~\eqref{DGL_Gen} using $\Lla$ as the driving measure. As in~\cite{wireless3}, it is unclear why this dependence should be continuous in $\Lla$. 

In order to cope with this problem, we introduce an approximating system of scalar differential equations where transmitters release connections only at discrete time steps. For this system, continuous dependence and unique existence of solutions can be established. Further, limiting trajectories of the approximations give rise to solutions of the original equation. Finally, using the tool of exponentially good approximations, we recover Theorem~\ref{LDP_NoSpatial} from the LDP for the approximating measures. 

\medskip
Let us start by verifying that the random measure $\tilde{\G}^\la$ as defined in \eqref{frustUserDefEmp} has the same distribution as $\G^\la$.

\begin{proposition}
	\label{markRepLem}
The random measures $\G^\la$ and $\tilde\G^\la$ have the same distribution.
\end{proposition}

In order to construct solutions of~\eqref{DGL_Gen} for general absolutely continuous driving measures, we introduce an approximating system of differential equations. This system corresponds to a scenario where transmitters release connections only at discrete time steps of size $\de>0$ such that the number of time steps is given by $\TF/\de \in \Z$. Before providing the detailed description of the system, we discuss the intuition behind the approximation. The system describes jointly the evolution of the normalized masses of
\begin{enumerate}
	\item guaranteed idle relays $\ai$, 
	\item guaranteed occupied relays $\ao{, k-1}$ {which get released in} the interval $\D_\de(k-1) = ((k-1)\de,k\de]$, and 
	\item critical relays $\ac{}$.
\end{enumerate}
At time zero all relays are idle, i.e.~$\ai_0 = 1$. After that, we describe the evolution of $\addi{\de}_t$ iteratively for $t \in \D_\de(k-1)$ as follows.
In the first approximating equation, the number of idle relays is reduced according to the mass the measure $\nu$. That is,
$$\ai_t = \ai_{(k-1)\de} - \int_{((k-1)\de, t]} \nu(\d s, ((k-1)\de, \TF], W, [0, \ai_{s-}]).$$
In particular, inside the interval $\D_\de(k-1)$ the idle relay mass $\ai_t$ is decreasing. At the interval boundary $k\de$ the idle relay mass increases by the mass of occupied relays $\ao{, k-1}_{k\de-}$ that leave in the time interval $\D_\de(k-1)$. In other words, 
$$\ai_{k\de} = \ai_{k\de-} + \ao{, k - 1}_{k\de-}.$$
At time zero no relays are occupied, so that $\ao{, j}_0 = 0$ for all $j \ge 0$. Next, chosen relays are counted as occupied if their exit times are not in the discretization window under consideration. This is captured by the equation
$$\ao{,j}_{t} = \ao{,j}_{(k-1)\de}+\int_{((k-1)\de,t]}\nu(\d s, \D_\de(j), W, [0, \ai_{s-}]),$$
where $j \ge k$.
Typically occupied relays with exit time in the interval $\D_\de(k-1)$ become idle by time $k\de$. However, this is no longer true if they are chosen again by another transmitter appearing in that interval. Hence, the mass of relays that can be released at time $k\de$ has to be decreased accordingly
		$$\ao{, k - 1}_t = \ao{, k - 1}_{(k - 1)\de}-\int_{((k - 1)\de, t]}\nu(\d s, ((k - 1)\de, \TF],W, [\ai_{s-}, \ai_{s-}+\ao{, k - 1}_{s-}]).$$
In order to quantify the loss of information caused by the discretization, we identify critical relays based on the discretization $\de$. If the transmitter's exit time is in the considered discretization window and hence entrance and exit times are in the same discretization, then the discretized picture provides only incomplete information. Therefore, such transmitters are counted as critical.
 Additionally, we count as critical the newly chosen relays which have been occupied prior to the time window with exit times in the time window. These two aspects give rise to the following equation for the critical relays
			\begin{align*}
				\ac_{t} =\ac_{(k-1)\de} &+ \int_{((k-1)\de,t]} \nu(\d s, \D_\de(k-1), W, [0, \ai_{s-} + \ao{,k-1}_{s-}])\\
				&+\int_{((k-1)\de, t]}\nu(\d s, (k\de,\TF],W, [\ai_{s-}, \ai_{s-} + \ao{,k-1}_{s-}]).
			\end{align*}
To summarize, we arrive at the following system of differential equations.
 
\begin{definition}
	Let $\nu\in\MM$ and define the following coupled system of differential equations with initial conditions $\ai_0 = 1$, $\ao{,j}_0 = 0$ and $\ac_0 = 0$.
	\begin{equation}\label{SysODE}
		\begin{split}
			\ai_t&=\ai_{(k-1)\de}-\int_{((k-1)\de,t]}\nu(\d s,((k-1)\de, \TF],W, [0, \ai_{s-}])\cr
			\ai_{k\de}&=\ai_{k\de-}+\ao{, k-1}_{k\de-}\cr
			\ao{, k-1}_{t}&=\ao{, k-1}_{(k-1)\de}-\int_{((k-1)\de, t]}\nu(\d s,((k-1)\de, \TF],W, (\ai_{s-}, \ai_{s-}+\ao{,k-1}_{s-}])\cr
			\ao{, k-1}_{k\de}&=0\cr
			\ao{, j}_{t}&=\ao{, j}_{(k-1)\de} + \int_{((k-1)\de, t]}\nu(\d s, \D_\de(j), W, [0, \ai_{s-}])\cr
			\ac_{t}&=\ac_{(k-1)\de}+\int_{((k-1)\de, t]}\nu(\d s, \D_\de(k-1),W,  [0, \ai_{s-}+\ao{, k-1}_{s-}])\cr
			&\phantom{=\ac_{(k-1)\de}}+\int_{((k-1)\de, t]}\nu(\d s, (k\de, \TF], W, ( \ai_{s-}, \ai_{s-}+\ao{, k-1}_{s-}]).
		\end{split}
	\end{equation}
	where $j\ge k$ and $t\in\D_\de(k-1)$.
\end{definition}

\medskip
In a first step, we establish existence and uniqueness of solutions of the above system for $\nu \in \Mac(\muT)$ which we then denote by $a(\nu)=(\ai(\nu), \{\ao{,j}(\nu)\}_{j\ge0}, \ac{}(\nu))$. To stress the dependence of the solution on the discretization parameter $\de$, we sometimes write $a^\de(\nu)=(\addi{\de}(\nu), \{\addo{\de}{,j}(\nu)\}_{j\ge0}, \addc{\de}{}(\nu))$.
\begin{proposition}
        \label{exUnAppProp}
	        Let $\nu\in\Mac(\muT)$, then the system~\eqref{SysODE} admits a unique solution.
\end{proposition}
By sending $\de \downarrow 0$, we arrive at a solution of the original equation~\eqref{DGL_Gen}. More precisely, define 
$$\b_t(\nu, r)= r - \limsup_{\de\downarrow0}  r\addi{\de}_t(r^{-1}\nu),$$ 
and 
$$\Memp(V) = \bigcup_{\r\ge0}\MM_\r(V)$$ as the union of empirical measures 
\begin{align}
    \label{mmrEq}
    \MM_\r(V)=\{\r\sum_{X_i\in X}\de_{X_i}:\,X\subset V, |X|<\infty \}
\end{align}
with weights $\r\ge 0$. Then, we have the following existence result.
\begin{proposition}
	        \label{approxScalProp}
		Let $\nu \in \Mac(\muT) \cup \Memp(V)$, then $\b_t(\nu, r)$ solves equation~\eqref{DGL_Gen}.
	\end{proposition}
	Having the scalar processes 
	$\addi{\de}(r^{-1} \nu)$ and $\b(\nu, r)$ at our disposal, we can now introduce the measures 
	$$\g^\de(\nu, r)(\d s, \d t, \d x) = \nu(\d s, \d t, \d x, [\addi{\de}_{s-}(r^{-1}\nu),1])$$
	and
	$$\g(\nu, r)(\d s, \d t, \d x) = \nu(\d s, \d t, \d x, [r^{-1}\b_s(\nu, r),1]).$$
In order to apply the exponential approximation machinery from~\cite[Theorem 4.2.23]{dz98}, three steps are required. First, we establish continuity of the function $\Mac(\muT) \to \MM(\itf^2 \times W)$, $\nu \mapsto \g^\de(\nu, r)$ as a function of $\nu$. For this we work in the $\tau$-topology both on the source and the target space. On the source space, it is the coarsest topology such that all evaluation maps $\nu \mapsto \nu(A)$, $A \in \mc{B}(V)=\{A\subset V: \, A \text{ is Borel measurable}\}$, are continuous. On the target space, it is the coarsest topology such that all evaluation maps $\g \mapsto \g(A)$, $A \in \mc{B}(\itf^2 \times W)$ are continuous.
\begin{proposition}
	        \label{contProp}
		The map $\nu \mapsto \g^\de(\nu, r)$ is continuous in the $\tau$-topology on $\Mac(\muT)$.
\end{proposition}

Second, we establish exponential approximation relations between $\G^\la$ and the approximating processes. Let $\Vert\cdot\Vert$ denote the total variational norm on the Banach space of finite signed measures, i.e., 
$$\Vert \g \Vert = \sup_{A \in \mc{B}(V)} |\g(A)|.$$
We start by considering the random measure of satisfied transmitters.
\begin{proposition}
	\label{odeRepProp1}
The random measure $\g^\de(\Lla, \rla)$ is an $\Vert\cdot\Vert$-exponentially good approximation of $\G^\la$.
\end{proposition}

The exponential approximation machinery is designed for random quantities that can be expressed as a functional of the empirical measure. Hence, as an intermediate step, we also replace $\rla$ by $r$.

\begin{proposition}
	\label{odeRepProp2}
The random measure $\g^\de(\Lla,r_\la)-\g^\de(\Lla,r)$ is an $\Vert\cdot\Vert$-exponentially good approximation of zero.
\end{proposition}

Third, the approximations $\g^\de(\nu,r)$ should be uniformly close to the true solution on sets of bounded entropy $\MM_\a(\muT) = \{\nu \in \MM:\, h(\nu|\muT) \le \alpha\}$. 
\begin{proposition}
	        \label{unifApproxProp}
		Let $\a, r > 0$ be arbitrary. Then,
		$$\lim_{\de \downarrow 0}\sup_{\nu\in\MM_\a(\mu)}\Vert\g^\de(\nu, r) - \g(\nu, r)\Vert = 0.$$
\end{proposition}

Using the above results, we can prove Theorem~\ref{LDP_NoSpatial}.

\begin{proof}[Proof of Theorem \ref{LDP_NoSpatial}]
Using Sanov's theorem as proved in \cite[Proposition 3.6]{wireless3} and the Propositions \ref{contProp}--\ref{unifApproxProp}, the result is a consequence of \cite[Theorem 1.13]{eiSchm}.
\end{proof}

\section{Proofs of Supporting results for Theorem~\ref{LDP_NoSpatial}}\label{thm1Sec}
In this section, we provide the proofs for Propositions \ref{markRepLem}-\ref{unifApproxProp}. First, in Section~\ref{auxSec}, we present auxiliary results, that we will use multiple times throughout the manuscript. Second, in Section~\ref{markovSec}, we derive a Markovian representation of the frustrated transmitters. Sections~\ref{existAppSec},~\ref{existSec} and~\ref{contSec} are devoted to existence, uniqueness and continuity properties of true and approximate solutions. Finally, in Sections~\ref{app1Sec} and~\ref{app2Sec}, we show that the approximate solutions are indeed close to the true ones.

%%%%%%%%%%%%%%%%%%%%%%%%%%%%%%%%%%%%%%%%
%%%%%%%%%%%%%%%%%%%%%%%%%%%%%%%%%%%%%%%%
\subsection{Auxiliary results}
\label{auxSec}
%%%%%%%%%%%%%%%%%%%%%%%%%%%%%%%%%%%%%%%%
%%%%%%%%%%%%%%%%%%%%%%%%%%%%%%%%%%%%%%%%
First, let us recall from~\cite[Lemma 3.1]{wireless3} some properties of absolutely continuous measures. 
\begin{lemma}\label{AbsoluteContinuity}
\begin{enumerate}
	\item	Let $\nu\in\Mac(\muT)$ be arbitrary. Then, $$\lim_{\e\downarrow0}\sup_{A\in \mathcal B(V):\, \mu_{\ms T}(A)<\e}\nu(A) = 0.$$
\item Let $\a>0$ be arbitrary. Then, $$\lim_{\e\downarrow0} \sup_{\substack{A\in \mathcal B(V):\, \mu_{\ms T}(A)<\e \\ \nu \in \MM_\a(\muT)}} \nu(A) = 0.$$
\item Let $\de>0$ be arbitrary and $N^{\e\la}$ be a random variable that is Poisson distributed with parameter $\e\la$
$$\lim_{\e\downarrow0}\limsup_{\la\uparrow\infty}\la^{-1} \log\P(N^{\e\la}>\la\de)=-\infty.$$
\end{enumerate}
\end{lemma}
\begin{proof}
	Part (1) rephrases the definition of absolute continuity. Part (2) can be shown using Jensen's inequality. 	Part (3) is a consequence of the Poisson concentration inequality~\cite[Chapter 2.2]{lugosi}. We refer the reader to~\cite[Lemma 3.1]{wireless3} for details.
\end{proof}
Next, we derive a simple yet powerful result on monotonicity of solutions of two specific differential equations. 
\begin{lemma}
	\label{solCurvesLem}
	Let $0 < a < a'$ and assume that $\nu \in \Mac(\muT)$ or $\nu \in \MM_\r(V)$ with $a - a' \in \rho\Z$. Let $A \in {\mc B}(\itf \times W)$, then the following holds.
	\begin{enumerate}
		\item If $b_t^a$ solves the equation
	$$b_t = a-\int_{0}^t\nu(\d s,A,[0, b_s]),$$
then $b_t^a \le b_t^{a'}$ for all $ t \le \TF$.
\item If $b_t^a$ solves the equation
	$$b_t = \int_{0}^t\nu(\d s,A,[0, a - b_s]),$$
then $b_t^a \le b_t^{a'}$ holds for all $t \le \TF$.
	\end{enumerate}
\end{lemma}
\begin{proof}
	For part (1), to derive a contradiction, assume that $b^{a}_t > b^{a'}_t$. 	Moreover, let $t_0 < t$ denote the last time before $t$ where $b^{a'}_{t_0} = b^{a}_{t_0}$. If $\nu$ is absolutely continuous, then the existence of $t_0$ follows from the continuity of the solutions $b^a_t$ and $b^{a'}_t$. If $\nu$ is an empirical measure, then $b^a_t$ and $b^{a'}_t$ are no longer continuous, but exhibit jumps of the same size $\rho$. In particular, the existence of $t_0$ follows from the assumption that $a - a' \in \rho\Z$.
	Then, 
	        $$b^{a'}_t = b^{a'}_{t_0} - \int_{[t_0, t)}\nu(\d s,  A,[0, b^{a'}_s]) \ge b^{a}_{t_0} - \int_{[t_0, t)}\nu(\d s,  A,[0,b^a_s]) = b^a_t,$$
		        which gives the desired contradiction. 

	For part (2) we argue similarly. More precisely, assume that $b^{a}_t > b^{a'}_t$. 
	Moreover, let $t_0 < t$ denote the last time before $t$ where $b^{a'}_{t_0}=b^{a}_{t_0}$. Then, 
		$$b^{a'}_t = b^{a'}_{t_0} + \int_{[t_0, t)}\nu(\d s, A, [0, a' - b^{a'}_s]) \ge b^{a}_{t_0} + \int_{[t_0, t)}\nu(\d s,  A,[0, a - b^a_s]) = b^a_t,$$
		        which again yields the desired contradiction. 
\end{proof}

The following lemma allows us to bound the discretization errors coming from sets of critical relays. 
\begin{lemma}
\label{stoDomLem}
Let $\{A^\de_*(\nu)\}_{\de, \nu}$ be a family of subsets of $V$ indexed by $\de > 0$ and $\nu \in \MM$. For $s \le \TF$ put 
$$A_s^\de(\nu) = \{(t, x, u):\, (s, t, x, u) \in A^\de_*(\nu)\}.$$
\begin{enumerate}
    \item If $\nu\in\Mac(\muT)$ and $\lim_{\de\downarrow0}\sup_{ s \le \TF}|A_s^\de(\nu)| = 0,$ then $\lim_{\de\downarrow0}\nu(A^\de_*(\nu)) = 0.$
    \item If $\lim_{\de\downarrow0}\sup_{\nu \in \MM_\a(\muT)}\sup_{ s \le \TF}|A_s^\de(\nu)| = 0$, then $\lim_{\de\downarrow0}\sup_{\nu \in \MM_\a(\muT)}\nu(A^\de_*(\nu)) = 0.$
    \item Assume that the process of sets $A_s^\de(r_\la^{-1}\Lla)$ is previsible. If $\sup_{ s \le \TF}|A_s^\de(r_\la^{-1}\Lla)|$ is an exponentially good approximation of zero, then $\Lla(A^\de_*(r_\la^{-1}\Lla))$ is an exponentially good approximation of zero.
\end{enumerate}
\end{lemma}
\begin{proof}
As for part (1) first note that by Lemma~\ref{AbsoluteContinuity} part (1), $\nu(A^\de_*(\nu))$ is arbitrarily close to zero if $\muT(A^\de_*(\nu))$ is sufficiently close to zero. Moreover, again by part (1) of Lemma~\ref{AbsoluteContinuity} using the absolute continuity of $\muT$ w.r.t.~the Lebesgue measure, $\muT(A^\de_*(\nu))$ becomes arbitrarily small if $|A^\de_*(\nu)|$ is sufficiently small. But $|A^\de_*(\nu)| \le \TF \sup_{ s \le \TF}|A_s^\de(\nu)|$ and the result follows.

Similarly for part (2), by Lemma~\ref{AbsoluteContinuity} part (1) and (2), $\sup_{\nu \in \MM_\a(\muT)}\nu(A^\de_*(\nu))$ is arbitrarily close to zero if $\sup_{\nu \in \MM_\a(\muT)} |A^\de_*(\nu)|$ is sufficiently close to zero. But 
$$\sup_{\nu \in \MM_\a(\muT)}|A^\de_*(\nu)|\le\TF \sup_{\nu \in \MM_\a(\muT)}\sup_{s \le \TF}|A_s^\de(\nu)|$$
 and the result follows. 

As for part (3) we want to prove that for all $\e>0$
\begin{align*}
\limsup_{\de \downarrow 0}\limsup_{\la \uparrow \infty}\la^{-1}\log\P(\Lla(A^\de_*(r_\la\Lla)) > \e)=-\infty.
\end{align*}
First, we condition on the number of users to be $n$ as well as on the ordered entrance times $(S_i)_{i\le n}$. Then, denoting $M_i = (T_i,X_i,U_i)$  and $Z_i = (S_i, M_i)$,
\begin{align*}
&\P\Big(\Lla(A^\de_*(r^{-1}_\la\Lla)) > \e \big|(S_i)_{i \le n}\Big)\cr
&=\P\Big(\#\{i \le n-1:\,Z_i\in A^\de_*(r^{-1}_\la\Lla)\} +\one\{Z_n\in A^\de_*(r^{-1}_\la\Lla)\}> \la\e\big|(S_i)_{i\le n}\Big)\cr
&=\E\Big[\E\Big(\P\Big(\#\{i \le n-1:\,Z_i\in A^\de_*(r^{-1}_\la\Lla)\} +\one\{Z_n\in A^\de_*(r^{-1}_\la\Lla)\}> \la\e\Big| (S_i)_{i\le n}\Big)\Big|(M_i)_{i<n}\Big)\Big]\cr
&=\E\Big[\E\Big(\P\Big(\#\{i \le n-1:\,Z_i\in A^\de_*(r^{-1}_\la\Lla)\} +\one\{M_n\in A^\de_{S_n}(r^{-1}_\la\Lla)\}> \la\e\Big| (S_i)_{i\le n}\Big)\Big|(M_i)_{i<n}\Big)\Big].
\end{align*}
Now, let $B_n$ be an independent uniform random variable on $[0,\mu_M(V')]$ with $\mu_M$ the normalized part of $\mu_{\ms{T}}$ acting on $[0,\TF]\times W\times [0,1]$.
Using previsibility and that, by independence, the probability for $M_n\in A^\de_{S_n}(r^{-1}_\la\Lla)$ is equal to the probability that $ B_n\le\mu_{M}(A^\de_{S_n}(r^{-1}_\la\Lla))$, we get that $\P(\Lla(A^\de_*(r^{-1}_\la\Lla)) > \e \big|(S_i)_{i \le n})$ is bounded above by
\begin{align*}
\P(\#\{i \le n-1:\,Z_i\in A^\de_*(r^{-1}_\la\Lla)\} + \one\{ B_n\le \mu_{M}(A^\de_{S_n}(r^{-1}_\la\Lla))\}> \la\e\big| (S_i)_{i\le n}).
\end{align*}
Hence, by induction, 
\begin{align*}
\P(\Lla(A^\de_*(r^{-1}_\la&\Lla)) > \e \big|(S_i)_{i< n}) = \P(\#\{ i \le n:\,  B_i\le \mu_{M}(A^\de_{S_i}(r^{-1}_\la\Lla))\}> \la\e\big| (S_i)_{i\le n})\cr
    &\le\P(\sup_{i\le n}\mu_{M}(A^\de_{S_i}(r^{-1}_\la\Lla))>\e'\big| (S_i)_{ i\le n})+\P(\#\{ i \le n:\,  B_i\le \e'\}> \la\e\big| (S_i)_{i\le n})
\end{align*}
for all $\e'>0$. As for the second summand, applying the Poisson point process expectation w.r.t.~the conditioning we have 
    $\P(\#\{i \le |X^\la|:\, B_i\le \e'\}> \la\e)$ where by independent thinning, $\#\{i \le |X^\la|:\,  B_i\le\e'\}$ is a Poisson random variable with intensity $\la\mu_{\ms{T}}(V)\e'$. Using part (3) of Lemma~\ref{AbsoluteContinuity}, as $\la$ tends to infinity, this summand has arbitrarily fast exponential decay for $\e'$ tending to zero. As for the first summand, we use our assumption and again part (3) of Lemma~\ref{AbsoluteContinuity}.
\end{proof}

The proof of Lemma~\ref{stoDomLem} reveals that part (3) remains true if $A^\de_* = A^{\de, \la}_*$ is allowed to depend on $\la$. The following result is the main application of Lemma~\ref{stoDomLem}. It shows that in the limit of small discretizations, critical users are negligible. 
\begin{lemma}
\label{scaleCritLem}
\begin{enumerate}
\item If $\nu \in \Mac(\muT) \cup \Memp(V)$, then $\lim_{\de \downarrow 0}\addc{\de}{}_{\TF}(\nu) = 0$. 
\item If $\alpha > 0$ is arbitrary, then $\lim_{\de \downarrow 0}\sup_{\nu \in \MM_\alpha(\mu)}\addc{\de}{}_{\TF}(\nu) = 0$.
\item The random measures $\addc{\de}{}_{\TF}(\rla^{-1}\Lla)$ form an exponentially good approximation of zero.
\end{enumerate}
\end{lemma}

\begin{proof}
First note that 
	$$\addc{\de}{}_{\TF}(\nu) = \nu(A_1^\de(\nu)) + \nu(A_2^\de(\nu)),$$
	where 
	$$A_1^\de(\nu) = \itf \times \D_\de(\lfloor s/\de\rfloor)\times W \times [0,\addi{\de}_{s-}(\nu)],$$
	and 
	$$A_2^\de(\nu) = \itf \times [(\lfloor s/\de\rfloor+1)\de, \TF] \times W\times [\addi{\de}_{s-}(\nu), \addi{\de}_{s-}(\nu) + \addo{\de}{,\lfloor s/\de\rfloor}_{s-}(\nu)].$$
	For part (1), note that if $\nu\in\Memp(V)$, then for $\de< \min_{i\ge1} (T_i-S_i)$ we have $\nu(A_1^\de(\nu))=0$ since $\nu(\{(s, t, x, u):\, (s, t, x, u) \in  \itf \times \D_\de(\lfloor s/\de\rfloor)\times W \times [0,1]\})=0$. Further, for sufficiently small $\de<\min_{i,j\ge1}|S_i-T_j|$ such that all entrance and exit times are well separated, also $\nu(A_2^\de(\nu))=0$.
	If $\nu\in\Mac(\muT)$, we show that both $A_1^\de(\nu)$ and $A_2^\de(\nu)$ satisfy the conditions of Lemma~\ref{stoDomLem} part (1). For $A_1^\de(\nu)$ this is clear, since 
$|A_{1,s}^\de(\nu)| \le \de |W|$ holds for every $s \le \TF$. For $A_2^\de(\nu)$ we have that 
	$$|A_{2,s}^\de(\nu)|\le \addo{\de}{,\lfloor s/\de\rfloor}_s(\nu) \le \nu(\itf \times \D_\de(\lfloor s/\de\rfloor)\times W \times [0,1]).$$
    Moreover,
	$$\lim_{\de \downarrow 0}\sup_{k \le \TF/\de - 1} |\itf \times\D_\de(k)\times W \times [0,1]| = 0,$$
so that the conditions of Lemma~\ref{stoDomLem} are satisfied. This finishes the proof of part (1). The last expression, using Lemma~\ref{stoDomLem} part (2), also implies that
$$\lim_{\de \downarrow 0}\sup_{k \le \TF/\de - 1}\sup_{\nu \in \MM_\a} \nu(\itf \times \D_\de(k) \times W\times [0,1]) = 0,$$ 
which completes the proof of part (2). To conclude the proof of part (3) observe that 	
\begin{align*}
\P(\sup_{s \le \TF} |A_{2,s}^\de(r_\la^{-1}\Lla)| > \e)&\le\P( \sup_{k \le \TF/\de - 1}\Lla(\itf \times \D_\de(k)\times W \times [0,1]) > \e r_\la^{-1})\cr
&\le\sum_{k \le \TF/\de - 1} \P( \Lla(\itf \times \D_\de(k)\times W \times [0,1]) > \e r_\la^{-1}),
	\end{align*}
    and hence by Lemma~\ref{AbsoluteContinuity} part (3), $\sup_{s \le \TF} \Lla(A_{2,s}^\de(r_\la^{-1}\Lla))$ indeed is an exponentially good approximation of zero.
\end{proof}

%%%%%%%%%%%%%%%%%%%%%%%%%%%%%%%%%%%%%%%%
%%%%%%%%%%%%%%%%%%%%%%%%%%%%%%%%%%%%%%%%
\subsection{Markovian representations}
\label{markovSec}
%%%%%%%%%%%%%%%%%%%%%%%%%%%%%%%%%%%%%%%%
%%%%%%%%%%%%%%%%%%%%%%%%%%%%%%%%%%%%%%%%
For the proof of Proposition~\ref{markRepLem} it will be convenient to consider the random measure $B$ of satisfied users defined by
$$B(\d s,\d t,\d x) = \Lla(\d s,\d t,\d x, [0,1]) - \G^\la(\d s,\d t,\d x).$$
\begin{proof}[Proof of Proposition~\ref{markRepLem}]
	The strategy of proof is to first condition on $\{(S_i, T_i, X_i)\}_{i \le N_\la}$ and then to show that the pair 
	$$\Big(\G^\la, \{B([0, t]\times[t, \TF]\times W)\}_{t \le \TF}\Big)$$
	has the same distribution as the pair 
	$$\Big(\tilde\G, \{\tilde B_t\}_{t \le \TF}\Big).$$
	For this note that, after the conditioning, both pairs become time-inhomogeneous Markov chains with jumps at times $S_i$ and $T_i$ of height $\la^{-1}$. Hence, it suffices to prove that the transition probabilities of the Markov chains coincide.
	
	Assume that there is an arrival $(S_i, T_i, X_i)$ at time $S_i$. In that case, there is a probability of 
	$1 - \rla^{-1}B([0, S_i-] \times [S_i,\TF] \times W)$ 
	of hitting an idle relay. If this happens, then $B([0, S_i-] \times [S_i,\TF] \times W)$ increases by $\la^{-1}$ and the random measure $\G^\la$ stays constant.
	Otherwise $B([0, S_i-] \times [S_i,\TF] \times W)$ stays constant and $\G^\la$ contains $(S_i, T_i, X_i)$ as an atom.
	Similarly, with probability $1 - \rla^{-1}\tilde B_{S_i-}$ the random variable $U_i$ is at most $1 - \rla^{-1}\tilde B_{S_i-}$. Then, $\tilde B_{S_i-}$ increases by $\la^{-1}$ and the random measure $\tilde\G^\la$ stays constant. Otherwise, $\tilde B_{S_i-}$ stays constant and the random measure $\tilde\G^\la$ contains $(S_i, T_i, X_i)$ as an atom.
	 At times $T_i$, in both cases, there is a deterministic decrease by $\la^{-1}$ if and only if the random measures contain $(S_i, T_i, X_i)$ as an atom.
\end{proof}

%%%%%%%%%%%%%%%%%%%%%%%%%%%%%%%%%%%%%%%%
%%%%%%%%%%%%%%%%%%%%%%%%%%%%%%%%%%%%%%%%
\subsection{Existence of a unique solution for the approximation}
\label{existAppSec}
%%%%%%%%%%%%%%%%%%%%%%%%%%%%%%%%%%%%%%%%
%%%%%%%%%%%%%%%%%%%%%%%%%%%%%%%%%%%%%%%%
Clearly, the system~\eqref{SysODE} has a unique solution if the driving measure is an empirical measure. In the large-deviation analysis of the high-density limit the empirical measures $\Lla$ are replaced by measures $\nu\in\Mac(\muT)$. Hence, we would expect that also the rare-event behavior of the derived quantity $\addi{\de}(\Lla)$, which is one component of the solution of~\eqref{SysODE}, can be expressed in terms of $\addi{\de}(\nu)$. In the next result, we show that $\addi{\de}(\nu)$ is well-defined if $\nu\in\Mac(\muT)$.
\begin{proof}[Proof of Proposition~\ref{exUnAppProp}]
First, existence and uniqueness only need to be verified for $t \mapsto \ai_t$ and $t \mapsto \ao{, k - 1}_t$ where we suppress the $\de$-dependence in the notation. Indeed, the remaining quantities are computed from them by explicit integration. By induction on $k$, it suffices to prove existence and uniqueness on each of the intervals $\D_\de(k-1)$, $k \ge 1$.

To begin with, we consider $\ai$. In order to work with increasing functions, we put
	$$\bi_t = \ai_{(k - 1)\de} - \ai_t,$$
	so that the differential equation becomes 
	$$\bi_t=\int_{(k-1)\de}^t\nu(\d s, ((k-1)\de, \TF], W, [0, \ai_{(k - 1)\de} - \bi_{s-}]).$$
	Moreover, introducing the measure 
	$$\wt{\nu}(\d s, \d u) = \nu(\d s, ((k-1)\de, \TF], W, \ai_{(k - 1)\de} \cdot \d u),$$ 
	this defining differential equation is transformed into
	$$\bi_t=\int_{(k-1)\de}^t\wt{\nu}(\d s, [0, 1 - \bi_{s-}/\ai_{(k-1)\de}]).$$
	In particular, the integral operator on the r.h.s.~is decreasing in $\bi_s$, so that existence and uniqueness of solutions are a consequence of~\cite[Proposition 2.2]{wireless3}.

	For $\ao{, k-1}$, we can proceed similarly. Indeed, after replacing $\nu$ by 
	$$\wt{\nu}(\d s, \d u) = \nu(\d s, ((k-1)\de, \TF], W, \ai_{s-} + \ao{,k-1}_{(k-1)\de} \cdot \d u)$$
	 existence and uniqueness of $\ao{,k-1}$ is again covered by~\cite[Proposition 2.2]{wireless3}.
\end{proof}

%%%%%%%%%%%%%%%%%%%%%%%%%%%%%%%%%%%%%%%%
%%%%%%%%%%%%%%%%%%%%%%%%%%%%%%%%%%%%%%%%
\subsection{Existence of solutions}
\label{existSec}
%%%%%%%%%%%%%%%%%%%%%%%%%%%%%%%%%%%%%%%%
%%%%%%%%%%%%%%%%%%%%%%%%%%%%%%%%%%%%%%%%
In this subsection, we show that taking the limit $\de\downarrow0$ in the approximating solutions $\g^\de(\nu,r)$ gives rise to a solution of the original system. First, we use Lemma~\ref{solCurvesLem} to show that the approximations are monotone w.r.t.~the discretization parameter $\de$. For this, we introduce the short hand notation $\nu(\D_\de(k))=\nu(\D_\de(k)\times\itf\times W\times[0,1])$.
\begin{lemma}
	\label{discMonLem}
	Let $\de>0$ and $\nu\in\Mac(\muT) \cup \Memp(V)$ be arbitrary and put $\de' = \de/2$. Then, for every $t \le \TF$ and $k \le \TF/\de - 1$
	\begin{enumerate}
		\item $\addi{\de'}_{t}(\nu) \ge \addi{\de}_t(\nu)$ and 
        \item $\sup_{n \ge 1}\sup_{t \le \TF}\big(\addi{2^{-n}\de}_t(\nu) - \addi{\de}_{t}(\nu)\big) \le \addc{\de}{}_{\TF}(\nu) + 2\sup_{l}\nu(\D_\de(l)).$
	\end{enumerate}
\end{lemma}
\begin{proof}
    To lighten notation, we suppress the $\nu$-dependence as well as the $W$-dependence in the notation in the proof. 
    First, we show that the asserted inequality in (2) is a consequence of part (1)  and 
	\begin{enumerate}
		\item[(1a)] $\addo{\de'}{,2j}_{k\de}(\nu) +\addo{\de'}{,2j+1}_{k\de}(\nu) \ge  \addo{\de}{,j}_{k\de}(\nu)$ for all $j, k \le \TF/\de - 1$.
	\end{enumerate}
Indeed, using part (1) and the fact that $1 - \addi{\de}_t = \addc{\de}{}_t + \sum_{j\ge0}\addo{\de}{,j}_t,$ we have
	\begin{align*}
\addi{2^{-n}\de}_t(\nu) - \addi{\de}_{t}(\nu)
&=(\addc{\de}{}_t -\addc{2^{-n}\de}{}_t) + \Big( \sum_{j\ge0}\addo{\de}{,j}_t-\sum_{j'\ge0}\addo{2^{-n}\de}{,j'}_t\Big).
	\end{align*}
Here, the first summand is bounded from above by $\addc{\de}{}_t\le \addc{\de}{}_{\TF}$. By part (1a), the second summand is bounded from above by 
\begin{align}\label{Est01}
\sum_{j\ge0}(\addo{\de}{,j}_t-\addo{\de}{,j}_{(k-1)\de}) + \sum_{j'\ge0}(\addo{2^{-n}\de}{,j'}_{(k-1)\de}-\addo{2^{-n}\de}{,j'}_t).
\end{align}
Note that by monotonicity inside the discretization, the second summand in \eqref{Est01} is bounded from above by 
\begin{align*}
\addo{2^{-n}\de}{,k-1}_{(k-1)\de}-\addo{2^{-n}\de}{,k-1}_t\le \nu(\D_\de(k-1)).
\end{align*}
Similarly, the first summand in \eqref{Est01} can be bounded from above by 
\begin{align*}
\sum_{j\ge k}(\addo{\de}{,j}_t-\addo{\de}{,j}_{(k-1)\de})\le \sum_{j\ge k}\nu(\D_\de(k-1)\times\D_\de(j)\times[0,1])\le \nu(\D_\de(k-1)).
\end{align*}
    Next, we prove (1) and (1a) by induction over $k$. That is, let us assume that part (1) holds for $t\le (k-1)\de$ and
part (1a) holds for $(k-1)\de$. First, part (1a) is trivial for $j< k$. If $j\ge k$, then part (1a) follows from the defining integral formula once part (1) is shown.

\medskip
For part (1), we consider the cases $t \in ((k-1)\de, (k-1/2)\de]$, $t \in ((k-1/2)\de, k\de)$ and $t = k\de$ separately. The case $t \in ((k-1)\de, (k-1/2)\de]$ is a consequence of Lemma~\ref{solCurvesLem} part (1) with $a=\addi{\de}_{(k-1)\de}$ and $a'=\addi{\de'}_{(k-1)\de}$. For $t \in ((k-1/2)\de, k\de)$ similar arguments apply. 
	Finally, assume that $t = k\de$, so that 
	$$\addi{\de}_{k\de} = \addi{\de}_{k\de-} + \addo{\de}{, k-1}_{k\de-}\quad\text{ and }\quad
	\addi{\de'}_{k\de} = \addi{\de'}_{k\de-} + \addo{\de'}{, 2k-1}_{k\de-}.$$
	We show, more generally, that for every $t \in ((k-1)\de, k\de)$,
	\begin{align}
		\label{discMonEq}
		\addi{\de'}_t + \addo{\de'}{, 2k-2}_t + \addo{\de'}{,2k-1}_t \ge \addi{\de}_t + \addo{\de}{,k-1}_t.
	\end{align}
	Indeed, if 
	$$\addi{\de'}_{(k-1)\de} + \addo{\de'}{, 2k-2}_{(k-1)\de} \ge \addi{\de}_{(k-1)\de} + \addo{\de}{,k-1}_{(k-1)\de},$$
	then, as in the case $t \in ((k-1)\de, (k-1/2)\de]$ considered above, we use Lemma~\ref{solCurvesLem} part (1). Otherwise, applying Lemma~\ref{solCurvesLem} part (1) inside the integral, for $t \in ((k-1)\de, (k-1/2)\de]$,
	\begin{align*}
		(\addi{\de'}_{t} + \addo{\de'}{,2k - 2}_{t}) - (\addi{\de'}_{(k-1)\de} + \addo{\de'}{,2k - 2}_{(k-1)\de}) 
		&= -\int_{((k-1)\de, t]}\nu(\d s, \itf, [0,\addi{\de'}_{s-} + \addo{\de'}{,2k - 2}_{s-}])\\
		&\ge -\int_{((k-1)\de, t]}\nu(\d s, \itf, [0,\addi{\de}_{s-} + \addo{\de}{,k - 1}_{s-}]) \\
		&= (\addi{\de}_{t} + \addo{\de}{,k-1}_{t}) - (\addi{\de}_{(k-1)\de} + \addo{\de}{,k-1}_{(k-1)\de}).
	\end{align*}
	In particular, by induction hypothesis,
	\begin{align*}
		\addi{\de'}_{t} + \addo{\de'}{, 2k-2}_{t} + \addo{\de'}{,2k - 1}_{t}  &\ge \addi{\de'}_{(k-1)\de} + \addo{\de'}{,2k - 2}_{(k-1)\de} + \addo{\de'}{,2k - 1}_{t} \\
													     &+ (\addi{\de}_{t} + \addo{\de}{,k-1}_{t}) - (\addi{\de}_{(k-1)\de} + \addo{\de}{,k-1}_{(k-1)\de})\ge \addi{\de}_{t} + \addo{\de}{,k-1}_{t}.
\end{align*}
Therefore, for $t\in ((k-1/2)\de, k\de)$ the assertion is again a consequence of Lemma~\ref{solCurvesLem} part (1). This completes the proof of~\eqref{discMonEq} and thereby of the lemma.
\end{proof}

Lemma~\ref{discMonLem} in particular implies that in the definition $\b_t(\nu, r)= r - \limsup_{\de\downarrow0}  r\addi{\de}_t(r^{-1}\nu)$ the limes superior is in fact a limit. 
We are now in the position to prove Proposition~\ref{approxScalProp}.
\begin{proof}[Proof of Proposition~\ref{approxScalProp}]
	First, note that 
		\begin{align*}
		\int_0^t\nu(\d s, [t, \TF],W,[0, \addi{\de}_{s-}(r^{-1}\nu)]) = r\sum_{j \ge 0} \addo{\de}{, j}_{t}(r^{-1}\nu) = r - r\addi{\de}_{t}(r^{-1}\nu) - r\addc{\de}{}_{t}(r^{-1}\nu).
	\end{align*}
	Hence, by monotone convergence, 
	\begin{equation*}
	\begin{split}
		\int_0^t\nu(\d s, [t, \TF],W,[0, 1 - \b_{s-}(\nu,r)/r]) &= \lim_{\de\downarrow0}  \int_0^t\nu(\d s,[t,\TF],W,[0,\addi{\de}_{s-}(r^{-1}\nu)])\cr
		&=\b_t(\nu, r) -r \lim_{\de\downarrow0}\addc{\de}{}_{t}(r^{-1}\nu).
	\end{split}
	\end{equation*}
	Now, Lemma~\ref{scaleCritLem} part (1) implies that $\lim_{\de\downarrow0}\addc{\de}{}_{t}(r^{-1}\nu)=0$ which 
	completes the proof.
\end{proof}

%%%%%%%%%%%%%%%%%%%%%%%%%%%%%%%%%%%%%%%%%%%%%%%%%%%%
%%%%%%%%%%%%%%%%%%%%%%%%%%%%%%%%%%%%%%%%%%%%%%%%%%%%
\subsection{Continuity for the approximation}
\label{contSec}
%%%%%%%%%%%%%%%%%%%%%%%%%%%%%%%%%%%%%%%%%%%%%%%%%%%%
%%%%%%%%%%%%%%%%%%%%%%%%%%%%%%%%%%%%%%%%%%%%%%%%%%%%
To prepare the proof of continuous dependence of the unique solution of~\eqref{SysODE} w.r.t.~the driving measure, we present two auxiliary results showing continuous dependence in simpler settings. 

\begin{lemma}
\label{contAuxLem0}
Let $a_\cdot(\cdot): \itf \times \Mac(\muT) \to [0,1]$ be a function such that 1) $\nu\mapsto a_s(\nu)$ is $\tau$-continuous for every $s \le \TF$ and 2) $s\mapsto a_s(\nu)\in[0,1]$ is piecewise continuous and monotone for every $\nu\in\Mac(\muT)$. Then, also the map 
$$\Phi:\, \nu \mapsto \nu(\d s, \d t, \d x, (a_s(\nu)+\d u)\cap[0, 1])$$
is continuous on $\Mac(\muT)$ where continuity is tested on sets of the form $A\times [a,b]$ with $A\in\mc{B}(\itf^2\times W)$ and $-1\le a<b \le 1$. 
\end{lemma}
\begin{proof}
    To prove the claim, we show that
$$\Big| \int_{A}\nu(\d s, \d t,\d x,[a + a_s(\nu), b + a_s(\nu)]) - \int_A\nu'(\d s, \d t,\d x, [a + a_s(\nu'), b + a_s(\nu')]) \Big|$$
becomes arbitrarily small for $\nu'$ sufficiently close to $\nu$. To simplify notation, we omit the integration symbols $\d t$ and $\d x$ in the rest of the proof. Introducing a mixed expression, it suffices to bound the following two 
contributions 
	\begin{equation}\label{Est2}
	\begin{split}
&\Big|\int_{A}(\nu-\nu')(\d s,[a + a_s(\nu), b + a_s(\nu)])\Big|\cr
&+\int_{A}\nu'(\d s,[a + a^-_s(\nu,\nu'), a + a^+_s(\nu,\nu')]\cup[b + a^-_s(\nu,\nu'), b + a^+_s(\nu,\nu')])
	\end{split}
\end{equation}
where $a^-(\nu,\nu')=a_s(\nu)\wedge a_s(\nu')$ and $a^+(\nu,\nu')=a_s(\nu)\vee a_s(\nu')$.
Let $\DDD$ be a partition of $\itf$ into intervals $\D_{\de'}(i)$ with mesh size $\de'>0$, which is compatible with the piecewise structure and write $I_i = \D_{\de'}(i)\times \itf\times W$. 

Then, the first summand in \eqref{Est2} can be bounded by 
	\begin{equation}\label{Est1}
	\begin{split}
&\sum_{i \in \DDD}\big|(\nu-\nu')(A\cap I_i,[a + a_{i\de'}(\nu), b + a_{i\de'}(\nu)])\big|\cr
&+\sum_{i \in \DDD}\int_{A\cap I_i}\nu(\d s,[a + a_s(\nu), a + a_{i\de'}(\nu)]\cup[b + a_s(\nu), b + a_{i\de'}(\nu)])\cr
&+\sum_{i \in \DDD}\int_{A\cap I_i}\nu'(\d s,[a + a_s(\nu), a + a_{i\de'}(\nu)]\cup[b + a_s(\nu), b + a_{i\de'}(\nu)]).
	\end{split}
\end{equation}
Moreover, by continuity of $a_s(\nu)$ w.r.t.~$s$, for sufficiently small $\de'$, we have $\sup_{i \in \DDD}|a_s(\nu)-a_{i\de'}(\nu)|<\e$ . Thus, the last two lines in~\eqref{Est1} can be bounded from above by 
	\begin{align*}
&2\sum_{i \in \DDD}\nu(A\cap I_i,[a + a_{i\de'}(\nu)-\e, a + a_{i\de'}(\nu)]\cup[b + a_{i\de'}(\nu)-\e, b + a_{i\de'}(\nu)])\cr
        &+\sum_{i \in \DDD}\big|(\nu'-\nu)\big(A\cap I_i \times \big ([a + a_{i\de'}(\nu)-\e, a + a_{i\de'}(\nu)]\cup[b + a_{i\de'}(\nu)-\e, b + a_{i\de'}(\nu)]\big)\big)\big|.
	\end{align*}
Since 
$$\sum_{i\in\DDD}\Big|A\cap I_i,[a + a_{i\de'}(\nu)-\e, a + a_{i\de'}(\nu)]\cup[b + a_{i\de'}(\nu)-\e, b + a_{i\de'}(\nu)])\Big| \le 2\e,$$ 
by part (1) of Lemma~\ref{AbsoluteContinuity}, the first term vanishes as $\de'$ tends to zero. Also the second term becomes arbitrarily close to zero for $\nu'$ sufficiently close to $\nu$. This also applies to the first line in \eqref{Est1}. 

In order to estimate the second contribution in \eqref{Est2}, we use similar arguments. Fix the same mesh size $\de'$ as above, and let $\nu'$ be sufficiently close to $\nu$, such that also $\sup_{i \ge 0}|a_{i\de'}(\nu)-a_{i\de'}(\nu')|<\e$. Then by piecewise monotonicity we can bound from above by,
\begin{align*}
&\sum_{i\in\DDD}\nu'(A\cap I_i,[a + a^-_{i\de'}(\nu,\nu'), a + a^+_{(i+1)\de'}(\nu,\nu')]\cup[b + a^-_{i\de'}(\nu,\nu'), b + a^+_{(i+1)\de'}(\nu,\nu')])\cr
&\le \sum_{i\in \DDD}\nu'(A\cap I_i,[a + a_{i\de'}(\nu)-\e, a+ a_{i\de'}(\nu)+2\e]\cup[b + a_{i\de'}(\nu)-\e, b+ a_{i\de'}(\nu)+2\e]).
\end{align*}
Again up to an arbitrarily small error, this is equal to 
\begin{align*}
\sum_{i\in \DDD}\nu(A\cap I_i,[a + a_{i\de'}(\nu)-\e, a+ a_{i\de'}(\nu)+2\e]\cup[b + a_{i\de'}(\nu)-\e, b+ a_{i\de'}(\nu)+2\e])
\end{align*}
for $\nu'$ sufficiently close to $\nu$ and 
$$\sum_{i\in\DDD}\Big|A\cap I_i,[a + a_{i\de'}(\nu)-\e, a+ a_{i\de'}(\nu)+2\e]\cup[b + a_{i\de'}(\nu)-\e, b+ a_{i\de'}(\nu)+2\e]\Big| \le 6|W|\TF^2\e.$$
Hence part (1) of Lemma~\ref{AbsoluteContinuity} concludes the proof.
\end{proof}

The following result allows us to deduce continuity of solutions of the approximating system of differential equations.
\begin{lemma}
\label{contAuxLem1}
Let $\de > 0$ be arbitrary and $\Phi:\, \Mac(\muT) \to \Mac(\muT)$ and $a:\, \Mac(\muT) \to [0,1]$ be continuous. Then, the solution $b(\nu)$ of the differential equation 
	\begin{align}\label{contAuxLem1_EQ}
b_t=\int_{(k-1)\de}^t\Phi(\nu)(\d s, A, W, [0, a(\nu) - b_s])
	\end{align}
	is continuous on $\Mac(\muT)$ for all $A\in\mc{B}([0, \TF])$.
\end{lemma}
\begin{proof}
	Let $\nu'\in\Mac(\muT)$. Then, we introduce an intermediate solution $b_t(\nu, \nu')$ of
	$$b_t=\int_{(k-1)\de}^t\Phi(\nu')(\d s, A, W, [0, a(\nu) - b_s]).$$
	First, by~\cite[Proposition 2.5]{wireless3}, 		$|b_t(\nu,\nu') - b_t(\nu)|$
	becomes arbitrarily small if $\nu'$ is sufficiently close to $\nu$, so that it remains to consider the deviation
$|b_t(\nu,\nu') - b_t(\nu')|$.
We claim that 
$$|b_t(\nu') - b_t(\nu,\nu')| \le |a(\nu') - a(\nu)|.$$ 
To prove this claim assume that $a(\nu) \le a(\nu')$, noting similar arguments are valid if the inequality is reversed. Then, part (2) of Lemma~\ref{solCurvesLem} shows that 
$$b_t(\nu') - b_t(\nu,\nu') \ge 0.$$
	 Applying part (1) of Lemma~\ref{solCurvesLem} to the trajectories $a(\nu') - b_t(\nu')$ and $a(\nu) - b_t(\nu,\nu')$	gives that 
	$$b_t(\nu') - b_t(\nu,\nu') \le a(\nu') - a(\nu),$$
	as required.
\end{proof}

Relying on Lemmas~\ref{contAuxLem0} and~\ref{contAuxLem1}, we now prove Proposition~\ref{contProp}.

\begin{proof}[Proof of Proposition~\ref{contProp}]
We start by establishing continuity of the scalar quantity  $\addi{\de}_t(\nu)$. Let $k\ge1$ be such that $t \in \D_\de(k-1)$ and assume that we have already established continuity of $\addi{\de}_{t'}(\nu)$ and $\{\addo{\de}{,j}_{t'}(\nu)\}_{j\ge0}$ for all $t' \le (k-1)\de$. 

Let $(k-1)\de<t < k\de$. To prove continuity of $\addi{\de}_t(\nu)$, note that Lemma~\ref{contAuxLem1}, with $a=\addi{\de}_{(k-1)\de}(\nu)$ and $\Phi$ the identity map, yield continuity of the solution $b_t(\nu)$ of the equation~\eqref{contAuxLem1_EQ}. But $\addi{\de}_{t}(\nu)=\addi{\de}_{(k-1)\de}(\nu)-b_t(\nu)$ and thus is also continuous. 

This also proves continuity of $\{\addo{\de}{,j}_{t}(\nu)\}_{j\ge k}$. Indeed, applying induction and Lemma~\ref{contAuxLem0} with $a=-1$, $b=0$ and $a_s(\nu) = \addi{\de}_s(\nu)$ shows that $\addo{\de}{,j}_{t}(\nu)$ is continuous in $\nu$. In order to prove continuity of $\addo{\de}{,k-1}_{t}(\nu)$, consider the integral equation 
\begin{align*}
b_t=\int_{(k-1)\de}^t\nu(\d s, A, W,[a_s(\nu),a_s(\nu) + a'(\nu) - b_s])=\int_{(k-1)\de}^t\Phi(\nu)(\d s, A, W,[0,a'(\nu) - b_s])
\end{align*}
where $a'(\nu)=\addo{\de}{,k-1}_{(k-1)\de}(\nu)$ is continuous by induction assumption and $\Phi$ is defined as in Lemma~\ref{contAuxLem0} with $a(\nu)=\addi{\de}(\nu)$ satisfying its assumptions. Thus, by  Lemma~\ref{contAuxLem1}, the solution $b_{t}(\nu)$ is continuous. But then $\addo{\de}{,k-1}_{t}(\nu)=\addo{\de}{,k-1}_{(k-1)\de}(\nu)-b_{t}(\nu)$ is also continuous.

For $t= k\de$, first note that $\addo{\de}{,k-1}_{k\de}(\nu)=0$ is continuous and also the mappings $\{\addo{\de}{,j}_{k\de}(\nu)\}_{j\ge k}=\{\addo{\de}{,j}_{k\de-}(\nu)\}_{j\ge k}$ are continuous. Further since $\addi{\de}_{k\de}(\nu)=\addi{\de}_{k\de-}(\nu)+\addo{\de}{,k-1}_{k\de-}(\nu)$ is a sum of continuous mappings, it is also continuous which completes the induction step.

\medskip
Finally, for the continuity for the measure valued process $\g^\de(\nu,r)$ note that 
\begin{align*}
\nu(\d s, \d t, \d x, [\addi{\de}_{s-}(\nu),1])=\sum_{k=0}^{\TF/\de}\nu(\d s, \d t, \d x, [\addi{\de}_{s-}(\nu),1])\one\{(k-1)\de\le s<k\de\}.
\end{align*}
Every summand is continuous by an application of Lemma~\ref{contAuxLem0} with $a=0$, $b=1$ and $a(\nu)=\addi{\de}(\nu)$ which finishes the proof.
\end{proof}

%%%%%%%%%%%%%%%%%%%%%%%%%%%%%%%%%%%%%%%%%%%%%%%%%%%%%%%%%%%%%%%%%%%%%%%%%%%%%%%%%%%
%%%%%%%%%%%%%%%%%%%%%%%%%%%%%%%%%%%%%%%%%%%%%%%%%%%%%%%%%%%%%%%%%%%%%%%%%%%%%%%%%%%
\subsection{Proof of Propositions~\ref{odeRepProp1} and~\ref{unifApproxProp}}
\label{app1Sec}
%%%%%%%%%%%%%%%%%%%%%%%%%%%%%%%%%%%%%%%%%%%%%%%%%%%%%%%%%%%%%%%%%%%%%%%%%%%%%%%%%%%
%%%%%%%%%%%%%%%%%%%%%%%%%%%%%%%%%%%%%%%%%%%%%%%%%%%%%%%%%%%%%%%%%%%%%%%%%%%%%%%%%%%
In this subsection, we prove that the solutions of the approximating system~\eqref{SysODE} give rise to good approximations to the true process of frustrated transmitters -- even when measured in a strong topology such as total variation distance. More precisely, we show the exponentially good approximation property for empirical measures (Proposition~\ref{odeRepProp1}) and uniform approximation on sets of bounded entropy (Proposition~\ref{unifApproxProp}).

\begin{proof}[Proof of Propositions~\ref{odeRepProp1} and~\ref{unifApproxProp}]
First, let $\nu \in  \Mac(\muT) \cup \Memp(V)$ be arbitrary. Now, monotonicity in $\delta$ gives that 
\begin{align*}
    &\Vert\g(\nu, r) - \g^{\de}(\nu, r)\Vert \\
    &\quad= \lim_{\de' \downarrow 0}\Vert\g^{\de'}(\nu, r) - \g^{\de}(\nu, r)\Vert \cr
    &\quad= \lim_{\de' \downarrow 0}\sup_{A \in \mc{B}(\itf^2\times W)} \int_{A \times [0,1]}\one\{\addi{\de}_{s-}(r^{-1}\nu) \le u \le \addi{\de'}_{s-}(r^{-1}\nu)\}  \nu(\d (s, t, x, u))\\% \nu(A, [0,\addi{\de'}_{s-}(r^{-1}\nu)]) - \nu(A, [0,\addi{\de}_{s-}(r^{-1}\nu)])|\\
	&\quad\le \nu(A_*^\de(\nu)),
\end{align*}
 where
    $$A_*^\de(\nu) = \{(s, t, x, u) \in V:\, u\in  [\ai_{s-}(r^{-1}\nu), \addi{\de}_{s-}(r^{-1}\nu)]\}.$$
Note that by part (2) of Lemma~\ref{discMonLem},
    $$ \sup_{s \le \TF}(\ai_{s-}(r^{-1}\nu) - \addi{\de}_{s-}(r^{-1}\nu))\le \addc{\de}{}_{\TF}(r^{-1}\nu) + 2\sup_{l}\nu(\D_\de(l))$$
so that the result follows from Lemma~\ref{stoDomLem} part (2) and part (3).
\end{proof}

%%%%%%%%%%%%%%%%%%%%%%%%%%%%%%%%%%%%%%%%%%%%%%%%%%%%%%%%%%%%%%%%%%%%%%%%%%%%%%%%%%%
%%%%%%%%%%%%%%%%%%%%%%%%%%%%%%%%%%%%%%%%%%%%%%%%%%%%%%%%%%%%%%%%%%%%%%%%%%%%%%%%%%%
\subsection{Proof of Proposition~\ref{odeRepProp2}}
\label{app2Sec}
%%%%%%%%%%%%%%%%%%%%%%%%%%%%%%%%%%%%%%%%%%%%%%%%%%%%%%%%%%%%%%%%%%%%%%%%%%%%%%%%%%%
%%%%%%%%%%%%%%%%%%%%%%%%%%%%%%%%%%%%%%%%%%%%%%%%%%%%%%%%%%%%%%%%%%%%%%%%%%%%%%%%%%%

Next, we need to show that we may replace $\rla^{-1}$ by $r$. More precisely, we claim that 
$$\g^\de(\Lla,r) - \g^\de(\Lla,\rla)$$
is an exponentially good approximation of zero in total variation distance. To achieve this goal, we introduce a refinement of the approximation defined by~\eqref{SysODE}. This refinement takes into account not only uncertainties in the time dimension, but also uncertainties with respect to the relay number. Loosely speaking, the approximations are built on the idea that for $r>\rla$, idle relays are reduced with rate $r^{-1}\ai{}$, whereas occupied relays are generated only at rate $\rla^{-1}\ai$. More precisely, we introduce the following system of differential equations.

\begin{definition}
Let $\rho>1$ and $\nu$ be an empirical measure. Then, 
        \begin{equation}\label{SysODERen}
                \begin{split}
                        \ai_t&=\ai_{(k-1)\de}-\int_{[(k-1)\de, t]}\nu(\d s,[0,\TF],W, [0, \ai_{s-}])\cr
                        \ai_{k\de}&=\ai_{k\de-} + \ao{,k-1}_{k\de-}\cr
                        \ao{,k-1}_{t}&=\ao{,k-1}_{(k-1)\de}-\int_{[(k-1)\de, t]}\nu(\d s, [0, \TF],W,[\ai_{s-},\ai_{s-}+\ao{,k-1}_{s-}])\cr
                                                \ao{,k-1}_{k\de}&=0\cr
                                                                        \ao{,j}_{t}&=\int_{[0,t]}\nu(\d s,\D_\de(j),W,[0, \rho^{-1}\ai_{s-}])\cr
                        \ac_{t}&=\ac_{(k-1)\de}+\int_{[(k-1)\de, t]}\nu(\d s,\D_\de(k-1), W,[0, \ai_{s-} + \ao{,k-1}_{s-}])\cr
&\hspace{1.5cm}+\int_{0}^{t}\nu(\d s,[k\de, \TF],W,[\ai_{s-},\ai_{s-} + \ao{,k-1}_{s-}])\cr
                        \acn_{t}&=\int_{[0,t]}\nu(\d s, \itf, W,[\rho^{-1}\ai_{s-}, \ai_{s-}])
                \end{split}
        \end{equation}
        where $j\ge k$ and the initial condition is given by $\ai_0=1$ and all other quantities equal to zero.
\end{definition}
If $\nu$ is an empirical measure, then the system~\eqref{SysODERen} has a unique solution that we denote by 
$$\{\addi{\de}(\rho, \nu),\{ \addo{\de}{,j}(\rho, \nu)\}_{j\ge 0}, \addc{\de}{}(\rho, \nu), \addcn{\de}{}(\rho, \nu)\}.$$
As before, in the marked setting we then define
\begin{align}
\label{busyMarkEq}
\g^\de(\r, \nu)( \d s,\d t, \d x, \d u)  = \nu(\d s, \d t, \d x, [\addi{\de}_{s-}(\r,\nu),1]).
\end{align}

Our intuition is that $\addi{\de}(\rho, \nu)$ should capture relays that can be guaranteed to be idle in the face of uncertainties stemming from both time and normalization fluctuations. In particular, $\addi{\de}(\rho, \nu)$ should be smaller than both $\addi{\de}(\nu)$ and $\rho\addi{\de}(\rho^{-1}\nu)$ since in the latter approximations only time fluctuations are taken into account. The next result provides a rigorous argument showing that this intuition is correct.

For the proof of Proposition~\ref{odeRepProp2}, we extend the strategy implemented for the derivation of Proposition~\ref{odeRepProp1}. More precisely, in Lemma~\ref{appDomCoupRenLem} we first make use of the concept of critical relays to provide a rigorous upper bound for the error exhibited in Definition~\ref{SysODERen}. After that, we rely on Lemma~\ref{stoDomLem} to show that the critical relays are an exponentially good approximation of zero.

\begin{lemma}
        \label{appDomCoupRenLem}
Let $\rho>1$ and $\nu$ be an empirical measure. Then,
        \begin{enumerate}
        \item for every $t \le \TF$, we have $\addi{\de}_t(\rho, \nu) \le \addi{\de}_t(\nu) \wedge \rho\addi{\de}_t(\rho^{-1}\nu)$, 
        \item $\addi{\de}_t(\nu) - \addi{\de}_t(\rho, \nu) \le \addc{\de}{}_{t}(\rho, \nu) + \addcn{\de}{}_{t}(\rho, \nu)$, and 
	\item $\rho\addi{\de}_t(\rho^{-1}\nu) - \addi{\de}_t(\rho, \nu) \le \addc{\de}{}_{t}(\rho, \nu) + \addcn{\de}{}_{t}(\rho, \nu) + \rho - 1$.
        \end{enumerate}
\end{lemma}
\begin{proof}
We suppress the $\de$-dependence in the proof and show that
 \begin{enumerate}
                \item[(1)] for every $t \le \TF$, we have $\ai_t(\rho, \nu) \le \ai_t(\nu) \wedge \rho\ai_t(\rho^{-1}\nu)$ and
                \item[(1a)] for every $j,k \ge 0$, we have $\ao{,j}_{k\de}(\rho, \nu) \le \ao{,j}_{k\de}(\nu) \wedge \rho\ao{,j}_{k\de}(\rho^{-1}\nu)$
        \end{enumerate}
using induction on $k$, where $t \in ((k-1)\de, k\de]$ be arbitrary.

First, assume that $t \ne k\de$. Then, the inequality $\ai_t(\rho, \nu) \le \ai_t(\nu)$ follows from Lemma~\ref{solCurvesLem} applied with $a = \ai_{(k-1)\de}(\rho, \nu)$ and $a' = \ai_{(k-1)\de}(\nu)$. Similarly, the inequality $\rho^{-1}\ai_t(\rho, \nu) \le \ai_t(\rho^{-1}\nu)$ follows from Lemma~\ref{solCurvesLem} applied with $a = \rho^{-1}\ai_{(k-1)\de}(\rho, \nu)$ and $a' = \ai_{(k-1)\de}(\rho^{-1} \nu)$. From Lemma~\ref{solCurvesLem}, we also conclude that 
$$\ai_t(\rho, \nu) + \ao{,k-1}_t(\rho, \nu) \le (\ai_t(\nu) + \ao{,k-1}_t(\nu)) \wedge (\ai_t(\rho^{-1}\nu) + \ao{,k-1}_t(\rho^{-1}\nu)).$$ 
Hence, part (1) also holds at $t = k\de$. Part (1a) follows from part (1) by the defining integral formula for $\ao{,j}_{k\de}(\rho, \nu)$.

Part (2) follows from part (1a), since 
        \begin{align*}
            \ai_t(\nu) - \ai_t(\rho, \nu)  =  \Big(\sum_{j \ge 0}\ao{,j}_{t}(\rho, \nu) + \ac_{t}(\rho, \nu) + \acn_{t}(\rho, \nu)\Big) - \Big(\sum_{j\ge0}\ao{,j}_{t}(\nu) + \ac_{t}(\nu)\Big).
        \end{align*}
Similarly, we can represent the difference $\rho\ai_t(\rho^{-1}\nu) - \ai_t(\rho, \nu)$ as 
        \begin{align*}
            \Big(\sum_{j\ge0}\ao{,j}_{t}(\rho, \nu) + \ac_{t}(\rho, \nu) + \acn_{t}(\rho, \nu)\Big) - \rho\Big(\sum_{j\ge0}\ao{,j}_{t}(\rho^{-1}\nu) + \ac_{t}(\rho^{-1}\nu)\Big) + \rho - 1,
        \end{align*}
so that an application of part (1a) concludes the proof of part (3).
        \end{proof}

Next, we note that as in Lemma~\ref{scaleCritLem} the number of users $\addc{\de}{}(\rho, \nu)$ that are critical due to time discretization vanish in the limit $\de \downarrow 0$.
        \begin{lemma}
                \label{critVanCoupRenLem}
Put $r^-_\la = r \wedge \rla$, $r^+_\la = r \vee \rla$ and $\rho_\la = r^+_\la/r^-_\la$. Then, $\addc{\de}{}_{\TF}(\rho_\la,(r^-_\la)^{-1}\Lla)$ is an exponentially good approximation of zero.
        \end{lemma}
        \begin{proof}
                Since the arguments from Lemma~\ref{scaleCritLem} apply verbatim, we omit the proof.
        \end{proof}
        Moreover, also second-type critical users $\addcn{\de}{}(\rho, \nu)$ become negligible as $\de \downarrow 0$.
        \begin{lemma}
                \label{critVanCoupRen2Lem}
It holds that  $\addcn{\de}{}_{\TF}(\rho_\la,(r^-_\la)^{-1}\Lla)$ is exponentially equivalent to zero.
        \end{lemma}
\begin{proof}
Since $\addi{\de}(\rho_\la,(r^-_\la)^{-1}\Lla)$ is bounded above by 1,
$$\limsup_{\la \uparrow\infty}\sup_{t \in \itf}|1-\rho_{\la}|\addi{\de}_{t-}(\rho_\la,(r^-_\la)^{-1}\Lla) \le \limsup_{\la \uparrow\infty} |1-\rho_{\la}| = 0.$$
In particular, the asserted exponential equivalence is a consequence of part (3) of Lemma~\ref{stoDomLem}.
\end{proof}

\begin{corollary}
        \label{scaleBetRenLem}
        The expressions
$$\addi{\de}_{t}(r^{-1}\Lla) - \addi{\de}_{t}(\rhola,(\rla^-)^{-1}L_\la)\qquad\text{ and }\qquad\addi{\de}_{t}(\rla^{-1}\Lla) - \addi{\de}_{t}(\rhola,(\rla^-)^{-1}L_\la).$$
are both of exponentially good approximations of 0.

\end{corollary}
\begin{proof}
We only prove the first assertion, as the second one is shown using similar arguments. 
If $\rla > r$, then part (2) of Lemma~\ref{appDomCoupRenLem} gives the upper bound for
$$|\addi{\de}_{t}(r^{-1}\Lla) - \addi{\de}_{t}(\rhola, r^{-1}L_\la)| \le \addc{\de}{}_{t}(\rho_\la,r^{-1}\Lla) + \addcn{\de}{}_{t}(\rho_\la,r^{-1}\Lla).$$
Hence, in that case Lemmas~\ref{critVanCoupRenLem} and~\ref{critVanCoupRen2Lem} conclude the proof. 
Similarly, if $r > \rla$, then part (3) of Lemma~\ref{appDomCoupRenLem} gives that 
$$|\rhola\addi{\de}_{t}(r^{-1}\Lla) - \addi{\de}_{t}(\rhola,(\rla^-)^{-1}\Lla)| \le \addc{\de}{}_{t}(\rho_\la,\rla^{-1}\Lla) + \addcn{\de}{}_{t}(\rho_\la,\rla^{-1}\Lla) + \rhola - 1,$$
so that another application of Lemmas~\ref{critVanCoupRenLem} and~\ref{critVanCoupRen2Lem} concludes the proof. 
\end{proof}

\begin{proof}[Proof of Proposition~\ref{odeRepProp2}]
As in the proof of Proposition~\ref{unifApproxProp}, we see that 
    \begin{align*}
        \Vert\g^\de(\Lla,r) - \g^\de(\Lla, \rla)\Vert &\le \Vert\g^\de(\Lla,r) - \g^\de(\rhola, (\rla^-)^{-1}\Lla)\Vert + \Vert \g^\de(\rhola, (\rla^-)^{-1}\Lla)  - \g^\de(\Lla, \rla)\Vert\\ 
        &\le \Lla(A^{(1), \de}_*(\Lla)) + \Lla(A^{(2), \la, \de}_*(\Lla)),
    \end{align*}
    where 
    $$A_*^{(1), \la, \de}(\Lla) = \{(s, t, x, u) \in V:\, u \in [\addi{\de}_{t}(\rhola, (\rla^-)^{-1}\Lla), \addi{\de}_{t}(r^{-1}\Lla) ]\}$$
    and 
    $$A_*^{(2), \la, \de}(\Lla) = \{(s, t, x, u) \in V:\, u\in [\addi{\de}_{t}(\rhola, (\rla^-)^{-1}\Lla), \addi{\de}_{t}(\rla^{-1}\Lla) ]\}.$$
Hence, applying Lemma~\ref{stoDomLem} part (3) together with Corollary~\ref{scaleBetRenLem} concludes the proof.
\end{proof}

\section{Outline of proof of Theorem~\ref{LDP_Spatial}}\label{Outline_Two}
Following the route of~\cite{wireless3}, we prove Theorem~\ref{LDP_Spatial} by reducing it to the setting of flat preference kernels considered in Theorem~\ref{LDP_NoSpatial}. Although in general, the preference kernel $\k$ is non-flat on a global scale, our assumptions imply that it can be approximated by a flat preference kernel locally. This allows us to apply Theorem~\ref{LDP_NoSpatial} on a local scale. In comparison to the setting in~\cite{wireless3}, the introduction of exit times entails that perturbations of the underlying point process can lead to more severe fluctuations in the process of frustrated users. Hence, more refined estimates are needed in order to derive the desired exponential approximation properties.

To make this precise,  we partition the given observation window $W$ into cubes $W^\de = \{W_1, \ldots, W_k\}$ of side length $\de$. Then, we introduce an approximating process as follows. We let a transmitter choose a sub-window $W_i$ according to the preference function, whereas the relay choice within $W_i$ is uniform. More precisely, put $\nuR = \muR$ or $\nuR = \lla$ and let $\Zld(\nuR)$ denote a Poisson point process on the state space $V(Y^\la)=\itf^2 \times W \times Y^\la$ with intensity measure $\la\mu^\de(\nuR, \lla)$ where 
$$\mu^\de(\nuR,\lla)(\d s, \d t, \d x, \d y) = \k^\de_{\nuR, \lla}(y|x)(\muST \otimes \lla)(\d s, \d t, \d x, \d y)$$
and, recalling $\k_{\nuR}$ from \eqref{Kappa_Index}, 
$$\k^\de_{\nuR, \lla}(y|x) = \sum_{i=1}^k \frac{\k_{\nuR}(W_i|x)}{\lla(W_i)}\one\{y \in W_i\}.$$

Note that our verbal description of the approximating process fits best to the process $\Zld(\lla)$. However, here the intensities of the locally flat preference kernels $\k_{\lla}(W_i|x)$ vary in $\la$, even after normalization, so that this setting is not covered by Theorem~\ref{LDP_NoSpatial}. This motivates the approximation $\Zld(\nuR)$ with $\nuR=\muR$ and with $\nuR=\lla$.

Now, note that $\Zld(\nuR)$ is a Poisson point process on the state space $V(Y^\la)$. In equation~\eqref{Busy_Process} we have seen how to construct an empirical measure of frustrated transmitters from such a Poisson point process. In the spatial situation we denote this process by $\g(\Lld(\nuR))$ where 
$$\Lld(\nuR) = \frac1\la \sum_{Z_i \in \Zld(\nuR)} \delta_{Z_i}.$$

To prove Theorem~\ref{LDP_Spatial}, we proceed in four steps. First, we leverage Theorem~\ref{LDP_NoSpatial} to establish an LDP for $\g(\Lld(\muR))$. As above, we denote by
$$\muT^\de(\muR, \muR) = \mu^\de(\muR, \muR) \otimes {\bf U}([0,1]),$$
the intensity measure on the extended state space $V'$.
\begin{proposition}\label{ExpEquiv_Spatial_m}
	The family of random measures $\g(\Lld(\muR))$ satisfies the LDP with good rate function $I^\de(\g) = \inf_{\nn \in \MM':\, \g(\nn) = \g} h(\nn|\muT^\de(\muR,\muR))$.
\end{proposition}
Second, it is possible to switch between $\nuR = \lla$ and $\nuR = \muR$ without changing substantially the approximating process of frustrated transmitters.

\begin{proposition}\label{ExpEquiv_Spatial}
The family of random measures $\g(\Lld(\muR)) - \g(\Lld(\lla))$ is $\Vert\cdot\Vert$-exponentially equivalent to zero.
\end{proposition}

Third, $\g(\Lld(\lla))$ is an exponentially good approximation of $\G^\la$.
\begin{proposition}\label{ExpEquiv_Spatial_2}
The family of random measures $\g(\Lld(\lla))$ is an $\Vert\cdot\Vert$-exponentially good approximations of $\G^\la$.
\end{proposition}

Finally,  after having established Propositions~\ref{ExpEquiv_Spatial_m},~\ref{ExpEquiv_Spatial} and~\ref{ExpEquiv_Spatial_2}, in Section~\ref{unifBoundSec} the proof of Theorem~\ref{LDP_Spatial} is completed by identifying the rate function.

\section{Proof of Theorem~\ref{LDP_Spatial} and its supporting results}\label{thm2Sec}

%%%%%%%%%%%%%%%%%%%%%%%%%%%%%%%%%%%%%%%%%%%%%%%%%%%%%%%%%
%%%%%%%%%%%%%%%%%%%%%%%%%%%%%%%%%%%%%%%%%%%%%%%%%%%%%%%%%
\subsection{Proof of Propositions~\ref{ExpEquiv_Spatial_m}}
%%%%%%%%%%%%%%%%%%%%%%%%%%%%%%%%%%%%%%%%%%%%%%%%%%%%%%%%%
%%%%%%%%%%%%%%%%%%%%%%%%%%%%%%%%%%%%%%%%%%%%%%%%%%%%%%%%%

In order to prove Proposition~\ref{ExpEquiv_Spatial_m}, we perform a reduction to the setting of flat preference functions considered in Theorem~\ref{LDP_NoSpatial}. 
\begin{proof}[Proof of Proposition~\ref{ExpEquiv_Spatial_m}]
	First, we decompose $\g(\Lld(\muR))$ into a sum of independent random measures
	$$\g(\Lld(\muR)) = \sum_{i \le k} \g(L_\la^{ \de, i}(\muR)),$$
	where $L_\la^{\de, i}(\muR)$ is the empirical measure associated with a Poisson point process on $V(Y^\la \cap W_i)$ with intensity measure 
	$$\la\frac{\k_{\muR}(W_i|x)}{\lla(W_i)} \one\{ y \in W_i\} (\muST \otimes  \lla)(\d s, \d t, \d x, \d y).$$
	Then, Theorem~\ref{LDP_NoSpatial} shows that $\g(L_\la^{ \de, i}(\muR))$ satisfies the LDP
	with good rate function 
	$$ \g \mapsto \inf_{\nu \in \MM: \, \g(\nu, \muR(W_i)) = \g} h(\nu|\mu_{\ms{T}, i}),$$
	where 
	$$\mu_{\ms{T}, i}(\d s, \d t, \d x, \d u) = \k_{\muR}(W_i|x)\mu_{\ms T}(\d s, \d t, \d x, \d u).$$
	Finally, by independence, we conclude from the identity  
$$\muT^\de(\muR,\muR) = \sum_{i \le k}\mu_{\ms{T}, i}(\d s, \d t, \d x, \d u)\otimes \frac{\one\{ y \in W_i\}\muR(\d y)}{\muR(W_i)} $$
 that $\g(\Lld(\muR))$ satisfies an LDP with good rate function 
 $$ \g \mapsto \inf_{\nn \in \MM': \, \g(\nn) = \g} h(\nn|\muT^\de(\muR,\muR)),$$
 as required.
 \end{proof}

%%%%%%%%%%%%%%%%%%%%%%%%%%%%%%%%%%%%%%%%%%%%%%%%%%%%%%%%%
%%%%%%%%%%%%%%%%%%%%%%%%%%%%%%%%%%%%%%%%%%%%%%%%%%%%%%%%%
\subsection{Proofs of Propositions~\ref{ExpEquiv_Spatial} and~\ref{ExpEquiv_Spatial_2}}

Lemma~\ref{totVarRandBoundLem} below shows that the total-variation distance between $\g(\Lld(\muR))$ and $\g(\Lld(\lla))$ can be computed in two steps. First, we determine the set of \emph{critical relays}. That is, those relays that are chosen in one of the processes but not the other. Second, we determine the number of transmitters pointing to the critical relays. 

More precisely, for any empirical measure $\nu$ on $V(Y^\la)$ and any relay $y\in Y$,
$$\pi(\nu)^{-1}(y) = \{(S_i, T_i, X_i, Y_i) \in \text{supp}(\nu):\, Y_i = y\}$$
denotes the set of all transmitters selecting relay $y$.  Then, 
	$$\Ycrit(\nu, \nu') = \{y \in Y:\, \pi(\nu)^{-1}(y) \ne \pi(\nu')^{-1}(y)\}$$
    denotes the set of \emph{critical relays}. In other words, non-critical relays must be chosen by the same transmitters in $\nu$ and in $\nu'$. Recalling the definition of $\MM_\r$ from~\eqref{mmrEq}, in the next result, we provide a concise bound on the total-variation distance between two transmitter processes in terms of critical relays. We denote by $\nu' \le \nu$ stochastic dominance of measures, that is $\nu' (A)\le \nu(A)$ for all $A\in \mathcal{B}(\hat V)$.
\begin{lemma}
	\label{totVarRandBoundLem}
	Let $\nu, \nu'\in \MM_\r$ be such that $\nu' \le \nu$. Then,
	$$\Vert \g(\nu) - \g(\nu') \Vert \le \nu(\pi(\nu)^{-1}(\Ycrit(\nu,\nu'))).$$
\end{lemma}
\begin{proof}
	First, the total-variation distance $\Vert \g(\nu) - \g(\nu')\Vert$ equals
	\begin{align*}
&\r\#(\text{supp}(\nu)\D\text{supp}(\nu'))\cr
&=\r\sum_{y\in \Ycrit(\nu, \nu')} \#(\pi^{-1}(\g(\nu))(y)\Delta \pi^{-1}(\g(\nu'))(y)) + \r\sum_{y\not\in \Ycrit(\nu, \nu')}\#(\pi^{-1}(\g(\nu))(y)\Delta \pi^{-1}(\g(\nu'))(y)).
\end{align*}
	$$$$ 
Clearly, the first summand is bounded above by $\nu(\pi(\nu)^{-1}(\Ycrit(\nu,\nu')))$. Hence, it remains to show that the second summand vanishes. In other words, we claim that a transmitter pointing to a non-critical relay is frustrated in $\nu$ if and only if it is frustrated in $\nu'$. 

To prove this claim, we perform induction on the arrival time of the transmitter, noting that the first transmitter pointing to a relay is always satisfied. Now, let $(S_i, T_i, X_i, Y_i) \in \nu$ be a transmitter pointing to a non-critical relay $Y_i$ and assume that we have proven the claim for transmitters arriving before $S_i$. By induction hypothesis, the relay $Y_i$ is already occupied at time $S_i$ in $\nu$ if and only if it is already occupied at time $S_i$ in $\nu'$. Therefore, also $(S_i, T_i, X_i, Y_i)$ is frustrated in $\nu$ if and only if it is frustrated in $\nu'$.
\end{proof}

 In order to compare the empirical measures $\Lld(\lla)$, $\Lld(\muR)$ and $L_\la$, it is essential to understand the differences in the intensity measures $\k^\de_{\lla, \lla}$, $\k^\de_{\muR, \lla}$ and $\k_{\lla}$. Therefore, we recall the following intensity bound from~\cite[Lemma 5.4]{wireless3}, where
$\mula = \musp \otimes \lla.$

\begin{lemma}
	\label{intBoundLem}
\begin{enumerate}
\item	Let $\de > 0$ be arbitrary. Then,
	$\lim_{\la \uparrow \infty}\int|\k^\de_{{\muR, \lla}} - \k^\de_{\lla, \lla}|\d \mula= 0.$
\item It holds that
	$\lim_{\de \downarrow 0}\lim_{\la  \uparrow \infty}\int|\k^\de_{\lla, \lla} -\k_{\lla}|\d \mula= 0.$
\item  It holds that
	$\lim_{\de \downarrow 0}\lim_{\la  \uparrow \infty}\int|\k^\de_{\muR, \muR} -\k_{\muR}|\d (\musp\otimes\muR)= 0.$
\end{enumerate}
\end{lemma}
\begin{proof}
	This is shown in~\cite[Lemma 5.4]{wireless3}.
\end{proof}

Now, we conclude the proof of Propositions~\ref{ExpEquiv_Spatial} and~\ref{ExpEquiv_Spatial_2}.

\begin{proof}[Proofs of Propositions~\ref{ExpEquiv_Spatial} and~\ref{ExpEquiv_Spatial_2}]
For the proof of Proposition~\ref{ExpEquiv_Spatial}, we let $\k^{\de,\ms{min}}$ denote the pointwise minimum of $\k^\de_{\muR, \lla}$ and $\k^\de_{\lla, \lla}$  and write $\Lmin$ for the empirical measure of the associated Poisson point process. In particular, 
	$$\Vert \g(\Lld(\muR)) - \g(\Lld(\lla)) \Vert \le \Vert \g(\Lld(\muR)) - \g(\Lmin) \Vert + \Vert \g(\Lmin) - \g(\Lld(\lla)) \Vert.$$
	We only derive a bound for the first summand, as we can proceed similarly for the second one.

	First, Lemma~\ref{totVarRandBoundLem} gives that 
	\begin{align*}
		\Vert \g(\Lld(\muR)) - \g(\Lmin) \Vert \le  \Lld(\muR)(\pi^{-1}(\Lld(\muR))(\Ycrit(\Lld(\muR),\Lmin))).
\end{align*}
By stochastic monotonicity, the right-hand side is bounded above by
\begin{align}
\label{expEquivEq}
\Lld(\muR)(\hat V)-\Lmin(\hat V) + \Lmin(\pi^{-1}(\Lmin)(\Ycrit(\Lld(\muR),\Lmin))).
\end{align}
For the first part, note that $\la(\Lla(\muR)(\hat V)-\Lmin(\hat V))$ is a Poisson random variable with parameter $\int|\k^\de_{\muR,\lla} - \k^{\de,\ms{min}}|\d {\mula}$. Hence, by part (1) of Lemma~\ref{intBoundLem},
	\begin{align}
\label{expEquivEq_1}
\limsup_{\la \uparrow\infty}\la^{-1}\log\P(\Lld(\muR)(\hat V)-\Lmin(\hat V) > \e ) = -\infty.
\end{align}
	For the second summand in~\eqref{expEquivEq}, we observe that $\Ycrit(\Lld(\muR),\Lmin)$ is measurable w.r.t.~$\Lld(\muR)-\Lmin$ and therefore independent of $\Lmin$. Hence, the expression 
	$$\la\Lmin(\pi^{-1}(\Lmin)(\Ycrit(\Lld(\muR),\Lmin)))$$ 
	is a Cox random variable with random intensity
	$$B^{\de, \la} = \int_{\itf^2\times W} \sum_{y \in \Ycrit(\Lld(\muR),\Lmin)} \k^{\de,\ms{min}}(y|x) (\muST)(\d s, \d t, \d x).$$
Since $\Ycrit(\Lld(\muR),\Lmin)$ is bounded from above by 
		$$\la (\Lld(\muR)(\hat V)-\Lmin(\hat V)),$$
		by \eqref{expEquivEq_1} we have 
$$\limsup_{\la \uparrow \infty} \la^{-1}\log\P(B^{\de, \la} > \e \la) = -\infty.$$
Moreover, 	
	\begin{align*}
&\P(\Lmin(\pi^{-1}(\Lmin)(\Ycrit(\Lld(\muR),\Lmin)))>\e)\le \P(N^{\k_\infty\muT(V)\e'\la}>\e\la)+\P(B^{\de, \la} > \e' \la)
\end{align*}
where $N^{\k_\infty\muT(V)\e'\la}$ denotes a Poisson random variable with parameter $\k_\infty\muT(V)\e'\la$. Hence, part (3) of Lemma~\ref{AbsoluteContinuity} concludes the proof Proposition~\ref{ExpEquiv_Spatial}.

\medskip
For the proof of Proposition~\ref{ExpEquiv_Spatial_2}, we proceed similarly. This time,  $\k^{\de,\ms{min}}$ is the pointwise minimum of $\k^\de_{\muR, \lla}$ and $\k_{\lla}$, so that 
	$$\Vert \g(\Lld(\muR)) - \G^\la \Vert \le \Vert \g(\Lld(\muR)) - \g(\Lmin) \Vert + \Vert \g(\Lmin) - \G^\la \Vert$$
and we can proceed as above, applying part (2) of Lemma~\ref{intBoundLem} instead of part (1).
\end{proof}

%%%%%%%%%%%%%%%%%%%%%%%%%%%%%%%%%%%%%%%%%%%%%%%%%%%%%%%%%
%%%%%%%%%%%%%%%%%%%%%%%%%%%%%%%%%%%%%%%%%%%%%%%%%%%%%%%%%
\subsection{Identification of the rate function and proof of Theorem~\ref{LDP_Spatial}}
\label{unifBoundSec}
%%%%%%%%%%%%%%%%%%%%%%%%%%%%%%%%%%%%%%%%%%%%%%%%%%%%%%%%%
%%%%%%%%%%%%%%%%%%%%%%%%%%%%%%%%%%%%%%%%%%%%%%%%%%%%%%%%%
Propositions~\ref{ExpEquiv_Spatial_m} -~\ref{ExpEquiv_Spatial_2} imply already that $\G^\la$ satisfies an LDP, but we do not know yet whether the rate function is of the form asserted in Theorem~\ref{LDP_Spatial}. In order to apply the machinery from~\cite[Theorem 4.2.23]{dz98},
we need uniform bounds on the total-variation distance of frustrated transmitters and the approximating process in the space of absolutely continuous measures. 

To achieve this goal, we proceed in several steps. First, in Section~\ref{approxTimeMeasSec}, we introduce an extension of the approximating process considered in Definition~\ref{SysODE} that is capable of reflecting not only fluctuations in time but also in the measures. Next, in Section~\ref{couplSec}, we introduce a coupling construction allowing us to represent both the frustrated transmitters and their approximations as functions of a common coupling measure. Finally, these two ingredients are combined in Section~\ref{unifBoundSec2} to derive the desired uniform approximation bound.

\subsubsection{Approximation w.r.t.~both time and measure}
\label{approxTimeMeasSec}

In Section~\ref{Outline_One}, we have introduced the process of frustrated transmitters as a limit of carefully chosen time-discretized approximations. In the following, we construct approximations not only w.r.t.~time but also w.r.t.~the measure. 

\begin{definition}
Let $\nu, \nu' \in \Mac(\muT)$ be such that $\nu \le \nu'$. Then, we consider the following  system of differential equations
        \begin{equation}\label{SysODESpatial}
                \begin{split}
                        \ai_t&=\ai_{(k-1)\de}-\int_{(k-1)\de}^t\nu'(\d s,[0,\TF],W,[0,\ai_s])\cr
                        \ai_{k\de}&=\ai_{k\de-} + \ao{,k-1}_{k\de-}\cr
                        \ao{,k-1}_{t}&=\ao{,k-1}_{(k-1)\de}-\int_{(k-1)\de}^{t}\nu'(\d s,[0,\TF],W,[\ai_s,\ai_s+\ao{,k-1}_s])\cr
                        \ao{,j}_{t}&=\int_{0}^{t}\nu(\d s,\D_\de(j),W,[0,\ai_s])\cr
                        \ac_{t}&=\ac_{(k-1)\de}+\int_{(k-1)\de}^{t}\nu(\d s,\D_\de(k-1),W,[0,\ai_s + \ao{,k-1}_s])\cr
                    &\phantom{=}+\int_{0}^{t}\nu(\d s, [k\de, \TF],W,[\ai_s,\ai_s + \ao{,k-1}_s])\cr
                        \acn_{t}&=\int_{0}^{t}(\nu'-\nu)(\d s, \itf,W, [0,\ai_s + \ao{,k-1}_s])\cr
                        \ao{,k-1}_{k\de}&=0
                \end{split}
        \end{equation}
        where $j\ge k$ and the initial condition is given by $\ai_0=1$ and all other quantities equal to zero.
\end{definition}

If $\nu$ and $\nu'$ are absolutely continuous, then the system~\eqref{SysODESpatial} has a unique solution.

\begin{lemma}
        \label{spatialExLem}Let $\nu, \nu' \in \Mac(\muT)$ be such that $\nu \le \nu'$. Then, the system~\eqref{SysODESpatial} has a unique solution $(\ai(\nu,\nu'), \ao{,*}(\nu,\nu'), \ac{}(\nu,\nu'),\acn(\nu,\nu'))$.
\end{lemma}
\begin{proof}
Since Lemma~\ref{spatialExLem} can be shown along the lines of Proposition~\ref{exUnAppProp}, we omit the proof.
\end{proof}

Conceptually, $\ai(\nu, \nu')$ should capture the relays that can be guaranteed to be idle in the face of uncertainties stemming from both time and measure fluctuations. In particular, $\ai(\nu, \nu')$ should be smaller than both $\ai(\nu)$ and $\ai(\nu')$ since in the latter approximations only time fluctuations are taken into account. The next result provides a rigorous argument showing that this intuition is correct.

\begin{lemma}
        \label{appDomCoupLem}
        Let $\nu, \nu' \in \Mac(\muT)$ be such that $\nu \le \nu'$. Then,
        \begin{enumerate}
        \item for every $t \le \TF$, we have $\addi{\de}_t(\nu, \nu') \le \addi{\de}_t(\nu) \wedge \addi{\de}_t(\nu')$, and
        \item $\addi{\de}_t(\nu)\vee \addi{\de}_t(\nu')  - \addi{\de}_t(\nu, \nu') \le \addc{\de}{}_{t}(\nu, \nu') + \addcn{\de}{}_{t}(\nu, \nu')$.
\end{enumerate}
\end{lemma}
\begin{proof}
We suppress the $\de$-dependence in the proof. Let us first prove
 \begin{enumerate}
                \item[(1)] for every $t \le \TF$, we have $\ai_t(\nu, \nu') \le \ai_t(\nu) \wedge \ai_t(\nu')$ and
                \item[(1a)] for every $j,k \ge 0$, we have $\ao{,j}_{k\de}(\nu, \nu') \le \ao{,j}_{k\de}(\nu) \wedge \ao{,j}_{k\de}(\nu')$
        \end{enumerate}
by induction on $k$ and let $t \in ((k-1)\de, k\de]$ be arbitrary.

     For $t \ne k\de$, domination for $\ai_t(\nu, \nu')$ follows from Lemma~\ref{solCurvesLem}. From the same lemma, we conclude that domination holds for $\ai_t(\nu, \nu') + \ao{,k-1}_t(\nu, \nu')$. Hence, domination for $\ai_t(\nu, \nu')$ also holds at $t = k\de$.
     By the defining integral formula for $\ao{,j}_{k\de}(\nu, \nu')$, part (1a) is implied by part (1). 
      Since
        \begin{align*}
                \ai_{k\de}(\nu) - \ai_{k\de}(\nu, \nu') = (\ao{}_{k\de}(\nu, \nu') + \ac_{k\de}(\nu, \nu') + \acn_{k\de}(\nu,\nu')) - (\ao{}_{k\de}(\nu) + \ac{}_{k\de}(\nu)),
        \end{align*}
        and similar for $\ai_{k\de}(\nu')$, part (2) is an immediate consequence of part (1a). 
        \end{proof}

As in Lemma~\ref{scaleCritLem}, the number of users $\addc{\de}{}(\nu, \nu')$ that are critical due to the time discretization, vanish in the limit $\de \downarrow 0$.
        \begin{lemma}
                \label{critVanCoupLem}
Let $\nu, \nu' \in \Mac(\muT)$ be such that $\nu \le \nu'$.
 Then, $\lim_{\de \downarrow 0}\addc{\de}{}_{\TF}(\nu, \nu') = 0$.
        \end{lemma}
        \begin{proof}
                Since the arguments from Lemma~\ref{scaleCritLem} apply verbatim, we omit the proof.
        \end{proof}

\begin{corollary}
        \label{scaleBetLem}
Let $\nu, \nu' \in \Mac(\muT)$ be such that $\nu \le \nu'$.  Then, for every $t \le \TF$,
        \begin{align*}
|\b_t(\nu) - \b_t(\nu')| \le 2\Vert\nu - \nu'\Vert.
\end{align*}
\end{corollary}
\begin{proof}
First, by Lemma~\ref{appDomCoupLem} part (1),
        \begin{align*}
            |\b_t(\nu) - \b_t(\nu')| \le \limsup_{\de \downarrow 0}\big(\addi{\de}_t(\nu) - \addi{\de}_t(\nu, \nu')\big) + \limsup_{\de \downarrow 0} \big(\addi{\de}_t(\nu') - \addi{\de}_t(\nu, \nu')\big).
\end{align*}
We only prove the bound for the first summand. The proof for the second summand is the same.
Writing $k_\de(t)$ for the integer $k$ determined by $t \in \D_\de(k-1)$, absolute continuity of $\nu$ and $\nu'$ implies that
    $$\limsup_{\de \downarrow 0}\big(\addi{\de}_t(\nu) - \addi{\de}_t(\nu, \nu')\big) = \limsup_{\de \downarrow 0}\big(\addi{\de}_{k_\de(t)\de}(\nu) - \addi{\de}_{k_\de(t)\de}(\nu, \nu')\big).$$
Hence, by Lemma~\ref{appDomCoupLem} part (2), Lemma~\ref{critVanCoupLem} and the definition of $\addcn{\de}{}$,
    $$\limsup_{\de \downarrow 0}\big(\addi{\de}_t(\nu) - \addi{\de}_t(\nu, \nu')\big) \le \addcn{\de}{}_{t}(\nu, \nu') \le  \Vert\nu - \nu'\Vert,$$
as required.
\end{proof}

%%%%%%%%%%%%%%%%%%%%%%%%%%%%%%%%%%%%%%%%%%%%%%%%%%%%%%%%%
%%%%%%%%%%%%%%%%%%%%%%%%%%%%%%%%%%%%%%%%%%%%%%%%%%%%%%%%%
\subsubsection{Coupling}
\label{couplSec}
%%%%%%%%%%%%%%%%%%%%%%%%%%%%%%%%%%%%%%%%%%%%%%%%%%%%%%%%%
%%%%%%%%%%%%%%%%%%%%%%%%%%%%%%%%%%%%%%%%%%%%%%%%%%%%%%%%%
As mentioned in the introduction to this section, identifying the rate function with the technique of~\cite[Theorem 4.2.23]{dz98} involves showing that the contraction mappings defining the approximations are close to the  contraction mappings defining the original rate functions uniformly on sets of bounded entropy. However,~\cite[Theorem 4.2.23]{dz98} is applicable if both, the approximating rate functions as well as the target rate function, are given via contraction mappings applied to a common rate function. 
Indeed, although both $I$ and $I^\de$ are defined via contractions based on relative entropy functions, the corresponding a priori measures are different. In order to remove this obstacle, we proceed as in~\cite[Section 5.5]{wireless3} and introduce a suitable coupling construction.

To compare different measures on $V'$, we add an additional $[0, \infty]$-coordinate and introduce the coupling space
$$V^* =V'\times  [0, \k_\infty]= \itfiww \times [0, \k_\infty].$$
More precisely, given a measure $\nns \in \MM(V^*)$ and a measurable function $f:\,W^2 \to [0,\k_\infty]$, we construct a measure $\nns(f)$ on $V'$ defined by first restricting to the sub-level set 
$$M(f) = \{(s, t, x, y, u, v):\, v \le f(x,y) \}$$ 
and then forgetting the last coordinate.  For instance, this definition allows us to represent $\muT(\muR)$ and $\muT^\de(\muR,\muR)$ as 
$$\muT(\muR) = \mu^*(\k_{\muR})\qquad \text{ and }\qquad \muT^\de(\muR,\muR) = \mu^*(\k^\de_{\muR,\muR}),$$
where
$$\mu^* = \mut \otimes \musp \otimes \muR  \otimes {\bf U}[0,1]\otimes |\cdot|.$$

Using this coupling construction, we now provide concise representations of total-variation distances of measures in $V'$. Indeed, for arbitrary measurable, bounded functions $f,g:\, W^2 \to [0,\k_\infty]$, we use the identity $\nn^*(f) - \nn^*(g)=\nnsp(f,g) - \nnsm(f,g)$, where
$$\frac{\d \nnsp(f,g)}{\d \nn^*}(s,t, x, y, u, v) =  \one\{ g(x,y)\le v \le f(x,y)  \} $$
and
$$\frac{\d \nnsm(f,g)}{\d \nn^*}(s,t, x, y, u, v) =  \one\{ f(x,y)\le v \le g(x,y) \}.$$
Thus, the total variation distance between $\nns(f)$ and $\nns(g)$ becomes
\begin{align}
        \label{tvDefEq}
        \Vert\nns(f) - \nns(g)\Vert = \max\{\nnsp(f,g)(V^*), \nnsm(f,g)(V^*)\}.
\end{align}

%%%%%%%%%%%%%%%%%%%%%%%%%%%%%%%%%%%%%%%%%%%%%%%%%%%%%%%%%
%%%%%%%%%%%%%%%%%%%%%%%%%%%%%%%%%%%%%%%%%%%%%%%%%%%%%%%%%
\subsubsection{Identification of the rate function}
\label{unifBoundSec2}
%%%%%%%%%%%%%%%%%%%%%%%%%%%%%%%%%%%%%%%%%%%%%%%%%%%%%%%%%
%%%%%%%%%%%%%%%%%%%%%%%%%%%%%%%%%%%%%%%%%%%%%%%%%%%%%%%%%

Recall that our goal is to show that the rate function of the LDP for $\G^\la$ is given by 
\begin{align}
\label{trueRfEq}
I(\g)=\inf_{\nn\in \MM':\, \g(\nn)=\g}h(\nn|\muT(\muR)).
\end{align}
After the preparations of the previous subsections, the only core ingredient that is missing to apply~\cite[Theorem 4.2.23]{dz98} is the following uniform approximation result, where $\MM'_\alpha = \{\nn \in \MM(V^*):\, h(\nn|\mus) \le \alpha\}$.
\begin{lemma}
	\label{uniformLem}
	Let $\alpha>0$ be arbitrary. Then, $\lim_{\de \downarrow 0} \sup_{\substack{\nns \in \MM'_\a}}\Vert\g(\nn^*(\k_{\muR})) - \g(\nn^*(\k_{\muR,\muR}^\de))\Vert = 0$.
\end{lemma}

We sketch very briefly how Lemma~\ref{uniformLem} implies that the rate function is of the form asserted in Theorem~\ref{LDP_Spatial}. For details, the reader is referred to the proof of~\cite[Proposition 4.3]{wireless3}.
\begin{proof}[Proof of Theorem~\ref{LDP_Spatial}]
	From Propositions~\ref{ExpEquiv_Spatial} and~\ref{ExpEquiv_Spatial_2} we conclude that $\g(\Lld(\muR))$ form exponential good approximations of $\G^\la$. Moreover, by Proposition~\ref{ExpEquiv_Spatial_m}, the empirical measure $\g(\Lld(\muR))$ satisfies an LDP with rate function 
\begin{align}
\label{approxRfEq}
I^\de(\g) \mapsto \inf_{\nn\in \MM(V'):\, \g(\nn)=\g}h(\nn|\muT^\de(\muR,\muR)).
\end{align}
Thus, once the uniform approximation bound from Lemma~\ref{uniformLem} is shown, it remains to verify that the rate functions in~\eqref{trueRfEq} and~\eqref{approxRfEq} coincide with
$$\inf_{\nn^*\in \MM(V^*):\, \g(\nn^*(\k_{\muR}))=\g}h(\nn^*|\mu^*) \text{ and }\inf_{\nn^*\in \MM(V^*):\, \g(\nn^*(\k_{\muR,\muR}^\de))=\g}h(\nn^*|\mu^*) ,$$
respectively. This is achieved by an optimization over the coupling coordinate, see~\cite[Proposition 4.3]{wireless3}. 
\end{proof}

We conclude the paper by proving Lemma~\ref{uniformLem} along the lines of~\cite[Lemma 5.5]{wireless3}. Nevertheless, for the convenience of the reader, we reproduce the most important steps. 

\begin{proof}[Proof of Lemma~\ref{uniformLem}]
We first simplify notation and write $\k$  and $\k^\de$ for $\k_{\muR}$ and $\k^\de_{\muR, \muR}$, respectively. By definition of $\g$, we need to compare the measures
	$$\int_{W} \nns(\k^\de)_y(\d s, \d t, \d x, [1-\b_s(\nns(\k^\de)_y), 1]) \muR(\d y)$$
	and
	$$\int_{W} \nns(\k)_y(\d s, \d t, \d x, [1-\b_s(\nns(\k)_y), 1]) \muR(\d y).$$
Recall that, by absolute continuity, 
	$$\nns(\k)(\d s, \d t, \d x, \d y, \d u) = \nns(\k)_y(\d s,\d t, \d x, \d u)\muR(\d y)$$
	and 
	$$\nns(\k^\de)(\d s, \d t, \d x, \d y, \d u) = \nns(\k^\de)_y(\d s, \d t, \d x, \d u)\muR(\d y).$$

	We subdivide the comparison into providing bounds separately for 
	$$\Vert\int_{W} (\nns(\k)_y-\nns(\k^\de)_y)(\d s, \d t, \d x, [1-\b_s(\nns(\k^\de)_y), 1]) \muR(\d y)\Vert$$
	and
	$$\int_{W} \nns(\k)_y(\itf^2 \times W \times \I(1-\b_s(\nns(\k)_y),  1-\b_s(\nns(\k^\de)_y))) \muR(\d y)$$
	where $\mc{I}(x,y) = [x \wedge y, x \vee y]$.
		The first expression is bounded above by $\Vert \nns(\k) - \nns(\k^\de)\Vert$, and identity~\eqref{tvDefEq}, Lemma~\ref{AbsoluteContinuity} part (2) and Lemma~\ref{intBoundLem} part (3) yield that
	$$\lim_{\de \downarrow 0}\sup_{ \nns \in  \MM'_{\a}}\Vert\nns(\k) - \nns(\k^\de)\Vert = 0.$$
	By Corollary~\ref{scaleBetLem}, the second expression is bounded above by $\nns(\k)(C_{\nns, \de})$ where
	$$C_{\nns, \de} = \{(s, t, x, y, u):\, |u-1+\b_s(\nns(\k)_y)| \le 2\Vert\nns(\k)_y- \nns(\k^\de)_y\Vert\}.$$
	In particular, by Lemma~\ref{AbsoluteContinuity} part (2) it remains to show that 
	$\lim_{\de \downarrow 0}\sup_{\nns \in \MM'_{\a}} \mu(\muR)(C_{\nns, \de}) = 0.$
	For this, we note that reversing the disintegration of the relay measure gives that
	\begin{align*}
		\mu(\muR)(C_{\nns, \de}) &\le 4\mut(\itf^2)\int_{W^2}\Vert \nns(\k)_y - \nns(\k^\de)_y\Vert \k(y|x) (\musp \otimes \muR) (\d x, \d y) \\
		&\le  4\mut(\itf^2)\musp(W)\k_\infty \int_W \max\{\nnsp(\k, \k^\de)_y(V') , \nnsm(\k, \k^\de)_y(V')\} \muR(\d y) \\
		&\le 4\mut(\itf^2)\musp(W)\k_\infty (\nnsp(\k, \k^\de)(V^*) + \nnsm(\k, \k^\de)(V^*)).
	\end{align*}
	Hence, using Lemma~\ref{AbsoluteContinuity} part (2) and Lemma~\ref{intBoundLem} finishes the proof.
\end{proof}

\section*{Acknowledgments}
This research was supported by the Leibniz program \emph{Probabilistic Methods for Mobile Ad-Hoc Networks}, by LMU Munich's Institutional Strategy LMUexcellent within the framework of the German Excellence Initiative.

\bibliography{../../wias}
\bibliographystyle{abbrv}

\end{document}